\documentclass[review]{elsarticle_arxiv}
\usepackage[utf8]{inputenc}
\usepackage{geometry}
\usepackage{amsmath}
\usepackage{amsfonts}
\usepackage{float}
\usepackage{stmaryrd}
\usepackage{graphicx}
\usepackage{subcaption}
\usepackage[all]{xy}
\usepackage{tikz-cd}
\usepackage{hyperref}
\usepackage{cleveref}
\usepackage{multirow}
\usepackage{amsthm}
\usepackage{grffile}

\newcommand{\rmd}{\, {\rm d}} 

\newcommand{\BB}{\boldsymbol{B}}
\newcommand{\ee}{\boldsymbol{e}}
\newcommand{\uu}{\boldsymbol{u}}
\newcommand{\vv}{\boldsymbol{v}}
\newcommand{\ww}{\boldsymbol{w}}

\newcommand{\g}{\mathfrak{g}}
\newcommand{\rhoh}{\rho_h}
\newcommand{\sh}{s_h}
\newcommand{\Bh}{\BB_h}
\newcommand{\uh}{\uu_h}
\newcommand{\LL}{\mathcal{L}}
\newcommand{\R}{\mathbb{R}}
\newcommand{\kin}{\mathrm{kin}}
\newcommand{\E}[1]{\mathrm{E{#1}}}

\DeclareMathOperator{\curl}{curl}
\DeclareMathOperator{\Div}{div}
\DeclareMathOperator{\grad}{grad}

\newtheorem{theorem}{Theorem}
\newtheorem{proposition}[theorem]{Proposition}%
\newtheorem{remark}{Remark}%
\newtheorem{hypothesis}{Hypothesis}

\newtheorem{definition}{Definition}

\crefname{hypothesis}{hypothesis}{hypotheses}
\crefname{tab}{Table}{Tables}
\crefname{definition}{}{}

\begin{document}

\begin{frontmatter}

\title{Variational discretizations of ideal magnetohydrodynamics in smooth regime using finite element exterior calculus}

\author{Valentin Carlier\corref{cor1}}
\ead{vcarlier@ipp.mpg.de}
\author{Martin Campos-Pinto}
\ead{martin.campos-pinto@ipp.mpg.de}
\cortext[cor1]{Corresponding author}
\affiliation[1]{Max-Planck-Institut fur Plasmaphysik
Boltzmannstrasse 2, Garching b. Munchen, Germany}
\date{\today}

\begin{abstract}
We propose a new class of finite element approximations to ideal compressible magnetohydrodynamic equations in smooth regime. 
Following variational approximations developed for fluid models in the last decade, our discretizations are built via a discrete variational principle mimicking the continuous Euler-Poincaré principle, and to further exploit the geometrical structure of the problem, vector fields are represented by their action as Lie derivatives on differential forms of any degree. 
The resulting semi-discrete approximations are shown to conserve the total mass, entropy and energy of the solutions for a wide class of finite element approximations.
In addition, the divergence-free nature of the magnetic field is preserved in a pointwise sense and a time discretization is proposed, preserving those invariants and giving a reversible scheme at the fully discrete level. 
Numerical simulations are conducted to verify the accuracy of our approach and its ability to preserve the invariants for several test problems.
\end{abstract}

\end{frontmatter}

\section{Introduction}

The ability to perform stable and accurate simulations of magnetohydrodynamic (MHD) equations is of primal importance in several fields. These models are for instance central in astrophysics, as they describe an important part of the visible matter in the universe~\cite{goossens2003introduction}, and in geophysics where the dynamics of the earth core and magnetic field also involve magnetized fluid described by MHD equations~\cite{molokov2007magnetohydrodynamics}.
In the physics of magnetic fusion they play a key role to study the macroscopic stability of plasma equilibria~\cite{freidberg1982ideal}. 
As the latter field is the main motivation of this work, we shall focus on deriving approximations which behave well on smooth solutions -- thus excluding astrophysical applications -- and on ideal flows -- while geophysical flows are dissipative.

The classical approach to discretizing MHD equations relies on their hyperbolic formulation, with well-established numerical methods such as Finite Volume~\cite{cargo1997roe} or Discontinuous Galerkin schemes ~\cite{bohm2020entropy,fambri2018ader}. 
To preserve the divergence-free constraint these schemes use different techniques, for instance constrained transport~\cite{evans1988simulation,stone2008athena} or generalized Lagrange multiplier methods to clean the divergence~\cite{dedner2002hyperbolic}. 

Since our focus here is on smooth regimes, we follow a variational approach where MHD equations can be seen 
as extremal curves of a Lagrangian action functional under specific constraints corresponding to a Euler-Poincaré reduction, as described in \cite{holm1998euler,newcomb1961lagrangian,maj2017mathematical}.
Our main motivation lies in the very good conservation and stability properties that variational discretizations generally enjoy.
For perfect fluids a variational approach was originally explored in \cite{pavlov2011structure} at low order, then carried to more general models in~\cite{gawlik2011geometric}. High order versions using finite element discretizations were then proposed in~\cite{natale2018variational} for the incompressible Navier Stokes equation, and thoroughly studied in~\cite{gawlik2020conservative,gawlik2021structure,gawlik2021variational,gawlik2022finite}
where variational discretizations for incompressible fluids with variable mass, compressible fluids, incompressible and barotropic MHD were described and implemented.  

In this work we propose a new class of variational schemes for ideal compressible MHD equations, which directly follow from a discrete Euler-Poincaré principle where the action of discrete vector fields on differential forms is defined through Cartan's formula for the Lie derivative.
Driven by a geometrical formulation of the MHD equations, we use the Finite Element Exterior Calculus (FEEC) framework ~\cite{arnold_falk_winther_2006_anum,arnold2018finite} to discretize the differential forms as well as the vector fields, and
finite element projection operators to define the interior products underlying our discrete Lie derivatives. 
A key point is that the main conservation properties of our schemes hold for general choices of FEEC spaces and projection operators.

FEEC discretizations involve several spaces to represent different unknowns, related to their geometrical nature. They have been first developed and used for linear problems such as Maxwell's equations~\cite{Bossavit.1998.ap,hiptmair2002finite} where classical finite element methods fail to properly approximate solutions on domains with holes or reentrant corners, and later studied in a more general context to approximate
Hodge Laplace operators \cite{arnold_falk_winther_2006_anum,arnold2018finite}.
A central component of this approach is the preservation of calculus identities $\curl \grad =0$ and $\Div \curl = 0$ at a discrete level, which makes it very well suited for electromagnetism and MHD where divergence constraints such as $\Div \BB =0$ play a crucial role. In particular this theory provides a natural discretization of differential forms. 
However, the discretization of vector fields, as well as transport operators,
is not canonical and is subject to a numerical design choice. 
In \cite{heumann_stabilized_2013,heumann2015stabilized} for instance, stable advection operators 
have been proposed for the transport of FEEC differential forms of any degree, based on weak formulations 
to handle velocity fields with low regularity.
Here we follow~\cite{pavlov2011structure,natale2018variational,gawlik2021variational} where vector fields are discretized via their transport action on densities, and generalize this principle to the transport of arbitrary forms. 
Previous works applying the FEEC framework to non linear models include~\cite{palha_mass_2017,carlier2023mass} for the incompressible Navier-Stokes equation, \cite{lee2020mixed} for the (rotating) compressible Euler equation, and~\cite{gawlik2021structure,gawlik2022finite,hu2021helicity,hu2017stable,hu2019structure,hiptmair_splitting-based_2018} for the MHD equations. 
We note that \cite{gawlik2021structure,gawlik2022finite} are also based on a Lagrangian formulation. 

A novelty of our work is the use a (strong) Lie-Cartan formula to discretize the transport of differential forms of any degree, leading to a discrete Euler-Poincaré principle with advected parameters. 
This approach provides a variational framework that allows for a wide class of conforming finite element spaces and discrete Lie derivatives
which commute with the exterior derivative. 
As a result, this framework allows us to preserve several invariants such as total mass, energy, entropy at the semi-discrete level, 
as well as a strong divergence-free property for the magnetic field. 
We further devise a time discretization that preserves the semi-discrete invariants at the fully discrete level, and is fully reversible. 
The non-dissipative character of our method is a promising feature for the simulation of magnetic fusion devices,
where current codes often require a level of artificial dissipation for numerical stability that is higher than the actual values observed in fusion plasmas~\cite{nikulsin2022jorek3d}.
In addition we show that our scheme is able to achieve high order accuracy for simple test-cases, which seems to be a novelty.

The remaining of this work is organized as follows: in \cref{sec:review_var_prin} we briefly review the variational principle underlying the ideal MHD equations and we reformulate it to allow for a natural discretization with FEEC spaces. In \cref{sec:discretization} we describe our discretization in a general setting and prove our conservation results. We next specify some choices leading to one particular scheme, and in \cref{sec:numerics} we present numerical results obtained with the resulting method. Section \ref{sec:conclusion} gathers some concluding remarks.

\section{Review of the variational structure of ideal MHD equations}
\label{sec:review_var_prin}

We review in this section the variational principle underlying ideal magnetohydrodynamics. After recalling the equations in classical form, we state the action principle in Eulerian variables which leads us to consider the associated constrained variations. We then rewrite the Lagrangian using differential forms to give a unified approach to the advection equations and the constrained variations.

\subsection{Equations of the ideal Magnetohydrodynamics}

The ideal MHD system for smooth flows is given by :
\begin{subequations}
\label{eqn:MHD}
\begin{equation}
\label{eqn:MHD_mass}
\partial_t \rho + \Div(\rho \uu) = 0 ~,
\end{equation}
\begin{equation}
\label{eqn:MHD_mom}
\rho \partial_t \uu + \rho (\uu \cdot \nabla \uu) + \nabla p + \BB \times \curl \BB  = 0 ~, 
\end{equation}
\begin{equation}
\label{eqn:MHD_s}
\partial_t s + \Div(s \uu) = 0 ~,
\end{equation}
\begin{equation}
\label{eqn:MHD_B}
\partial_t \BB + \curl(\BB \times \uu) = 0 ~,
\end{equation}
\end{subequations}
where $\rho$ is the plasma density, $\uu$ the velocity field, $s$ the entropy density and $\BB$ the magnetic field.
Here \cref{eqn:MHD_mass} describes the transport of density, \cref{eqn:MHD_mom} is the evolution of the momentum (analogous to Newton's second law), \cref{eqn:MHD_s} is the transport of entropy (analogous to the second law of thermodynamics for a reversible system), and \cref{eqn:MHD_B} describes the evolution of the magnetic field  assuming that the plasma is a perfect conductor. Equations \eqref{eqn:MHD_mass} and \eqref{eqn:MHD_s} can be combined to express the evolution of internal energy or pressure, leading to a more common form for the MHD system. 
These equations which describe the evolution of a plasma by neglecting any resistive effect (either the magnetic resistivity modelling the imperfect conducting properties of a plasma, or the fluid viscosity modelling its internal friction) are a good first approximation for magnetic fusion plasmas.

\subsection{Lagrangian with proxy field}

Our variational discretization is based on the following classical result~\cite{newcomb1961lagrangian}, which states that solutions to the ideal MHD equations are extrema of an action functional. Given $\Omega$ a domain in $\R^n$ (we will consider here the cases $n=2,3$), we denote by $X(\Omega)$ the set of vector fields in $\Omega$ that are tangent to the boundary $\partial \Omega$. 

\begin{theorem}
\label{thm:mhd_var_lag}
Let $l$ be the following Lagrangian :
\begin{equation}
    \label{eqn:MHD_reduced_lagrangian}
    l(\uu,\rho, s, \BB)=\int_\Omega \frac{1}{2}\rho|u|^2-\rho e(\rho, s) -\frac{1}{2}|\BB|^2 dV ~,
\end{equation}
where $e$ is the internal energy depending on $\rho$ the density and $s$ the entropy, $\uu$ is the velocity of the fluid and $\BB$ the magnetic field.
Smooth solutions of the ideal MHD equations correspond to extremal curves of the corresponding action 
\begin{equation}
    \label{eqn:MHD_reduced_action}
    S(\uu,\rho, s, \BB) = \int_0^T l(\uu(t),\rho(t), s(t), \BB(t)) dt~,
\end{equation}
under constrained variations of the form 
\begin{equation} 
  \label{constraints}
    \delta \uu = \partial_t \vv + [\uu,\vv], 
    \qquad 
    \delta \rho = -\Div (\rho \vv),
    \qquad
    \delta s = -\Div (s \vv),
    \qquad
    \delta \BB = -\curl(\BB\times \vv)~,
\end{equation}
where $\vv = \vv(t)$ is a curve in $X(\Omega)$ which is null at both end-points.
Here the Lie bracket $[\cdot , \cdot]$ is defined by
\begin{equation}
\label{eqn:Lie-bracket}
[\uu, \vv] = \sum\limits_{i=1}^n \sum\limits_{j=1}^n (u_j \partial_j v_i - v_j \partial_j u_i) \ee_i = \uu \cdot \nabla \vv - \vv \cdot \nabla \uu ~,
\end{equation}
where $(u_i)$ (resp. $(v_i)$) are the components of $\uu$ (resp. $\vv$) in the canonical basis $(\ee_i)$ of $\R^n$.

This variational principle is supplemented with the advection equations
\begin{subequations}
\begin{equation}
\label{eqn:advection_rho}
\partial_t \rho + \Div(\rho \uu) = 0 ~,
\end{equation}
\begin{equation}
\label{eqn:advection_s}
\partial_t s + \Div(s \uu) = 0 ~,
\end{equation}
\begin{equation}
\label{eqn:advection_B}
\partial_t \BB + \curl(\BB \times \uu) = 0 ~
\end{equation}
\end{subequations}
and the constraint $\Div \BB = 0$. It is enough for this latter constraint to be satisfied at one given time, since \cref{eqn:advection_B} then guarantees its preservation at all times.
\end{theorem}
\begin{remark}
The form of the constrained variations, as well as the advection equations, come from the Euler-Poincaré reduction of a general Lagrangian on the diffeomorphism group, with advected parameters (see~\cite{holm1998euler, gawlik2011geometric}).
\end{remark}
\begin{proof}
Let $\uu$ be an extremal curve of the action $S = \int_0^T l(\uu, \rho, s, \BB)$ under the constraints \eqref{constraints}.
Then for every curve $\vv$ in $X(\Omega)$ we have :
\begin{equation} \label{EL}
        0 = \frac{\delta S}{\delta \uu} \delta \uu + \frac{\delta S}{\delta \rho} \delta \rho + \frac{\delta S}{\delta s} \delta s + \frac{\delta S}{\delta \BB} \delta \BB 
        = \int_0^T \frac{\delta l}{\delta \uu} \delta \uu + \frac{\delta l}{\delta \rho} \delta \rho + \frac{\delta l}{\delta s} \delta s + \frac{\delta l}{\delta \BB} \delta \BB.
\end{equation}
%
%
Using integration by parts we develop the first term as
$$
\begin{aligned}
    \int_0^T \frac{\delta l}{\delta \uu} \delta \uu 
        &=  \int_0^T \int_\Omega \rho \uu \cdot (\partial_t \vv + \uu \cdot \nabla \vv - \vv \cdot \nabla \uu)
        \\
        &=    -\int_0^T \int_\Omega \partial_t(\rho \uu) \cdot \vv + \rho (\uu \cdot \nabla \uu) \cdot \vv + \uu \Div(\rho \uu) \cdot \vv  + (\nabla \uu)^T (\rho \uu) \vv
\end{aligned}
$$
the second term as
$$
\begin{aligned}
    \int_0^T \frac{\delta l}{\delta \rho} \delta \rho
        &=  - \int_0^T \int_\Omega \Big( \frac{|\uu|^2}{2}-e(\rho, s)-\rho \partial_{\rho} e(\rho, s)\Big) \Div(\rho \vv)
        \\
        &=  \int_0^T \int_\Omega \nabla \Big( \frac{|\uu|^2}{2} -e(\rho, s)-\rho \partial_{\rho} e(\rho, s)\Big) \cdot (\rho \vv)
        \\
        &=  \int_0^T \int_\Omega \big( (\nabla \uu)^T (\rho \uu) - \nabla \rho \partial_{\rho} e(\rho, s) - \nabla s \partial_{s} e(\rho, s) - \nabla (\rho \partial_{\rho} e(\rho, s))\big) \cdot (\rho \vv)
\end{aligned}
$$
the third term as
$$
    \int_0^T \frac{\delta l}{\delta s} \delta s
        =  \int_0^T \int_\Omega \rho \partial_s e(\rho, s) \Div (s \vv)
        =  - \int_0^T \int_\Omega \nabla (\rho \partial_s e(\rho, s)) \cdot (s \vv)
$$
and the fourth term as
$$
    \int_0^T  \frac{\delta l}{\delta \BB} \delta \BB
        =  \int_0^T \int_\Omega  \BB \cdot \curl(\BB \times \vv)
        =  -\int_0^T \int_\Omega  (\BB \times \curl \BB) \cdot \vv ~.
$$
Introducing the pressure defined as $p=\rho(\rho \partial_{\rho} e+ s \partial_s e)$ we next observe that
$$
\rho \big(\nabla \rho \partial_{\rho} e(\rho, s) + \nabla s \partial_{s} e(\rho, s) + \nabla (\rho \partial_{\rho} e(\rho, s))\big) + s \nabla (\rho \partial_s e(\rho, s))
     = \nabla p~,
$$
so that \eqref{EL} being null for every $\vv$ yields

%

$$
0 
= \partial_t(\rho \uu) + \rho (\uu \cdot \nabla \uu) + \uu \Div(\rho \uu) + \nabla p + \BB \times \curl \BB ~. 
$$

Finally, developing $\partial_t(\rho \uu) = \uu \partial_t \rho + \rho \partial_t \uu$ and using \cref{eqn:advection_rho} gives us the momentum equation for ideal MHD in its usual form : 
\begin{equation}
\label{eqn:momentum_MHD}
\rho \partial_t \uu + \rho (\uu \cdot \nabla \uu) + \nabla p + \BB \times \curl \BB  = 0~.
\end{equation}
\end{proof}
\subsection{Lagrangian with differential forms}
\label{sec:lag_diff}
We now rewrite the above principle using differential forms: we refer to \cref{tab:proxy} and~\cite{heumann2015stabilized,hiptmair2002finite,arnold_falk_winther_2006_anum} 
for more details on differential forms and their proxy vectors.
\begin{table}
\begin{tabular}{|p{1.7cm}||p{1.4cm}|p{3.9cm}|p{5cm}|p{2.2cm}|}
 \hline
 $\omega \in \Lambda^k(\Omega)$ 
    & $k=0$ 
      & $k=1$ 
        & $k=2$ 
          & $k=3$ \\
  $\omega = \cdots$
   & $w$ 
     & $w_x dx + w_y dy + w_z dz$ 
       & $w_x dy \wedge dz + w_y dz \wedge dx + w_z dx \wedge dy$ 
         & $w \, dx \wedge dy \wedge dz$ \\
 \hline
 \hline
 $\omega$     & $w$ & $\ww = (w_x, w_y, w_z)$ & $\ww = (w_x, w_y, w_z)$ & $w$ \\
 \hline
 $d \omega$  & $\grad w$ & $\curl \ww$ & $\Div \ww$ & $0$ \\
 \hline
 $i_{\vv} \omega $ & 0 & $\ww \cdot \vv$ & $\ww \times \vv$ & $w \vv$ \\
 \hline
 $L_{\vv} \omega $ & $\vv \cdot \grad w$ & $\grad (\ww \cdot \vv) + (\curl \ww) \times \vv $ & $\curl(\ww \times \vv) + (\Div \ww) \vv $ & $\Div(w \vv)$ \\
 \hline

\end{tabular}
\caption{Correspondence between differential forms (in the lower-left column) and their vector proxies (in the lower-right table) for the operators used in this article, namely the exterior derivative, interior product and Lie derivative associated with a vector field $\vv$.}
\label{tab:proxy}
\end{table}
Denoting $\varrho$ for $\rho dV$, $\varsigma$ for $s dV$ and $\beta=B_x dy\wedge dz + B_y dz \wedge dx + B_z dx \wedge dy$, the reduced Lagrangian can be rewritten :
\begin{equation}
    \label{eqn:MHD_reduced_lagrangian_forms}
    l(\uu,\varrho, \varsigma, \beta)=\int_\Omega \frac{1}{2}\varrho|u|^2-\rho e(\varrho, \varsigma) -\frac{1}{2}\beta \wedge \star \beta ~,
\end{equation}
where $\star$ denotes the Hodge operator (which here maps the $(n-1)$ form $\beta$ to a 1 form). The ideal MHD system is recovered considering extremizor of the action of this lagrangian under constrained variations of the form (comes from a simple rewriting of \cref{thm:mhd_var_lag}) : 
\begin{equation}
    \delta \uu = \partial_t \vv + [\uu,\vv] ~, 
    \qquad 
    \delta \rho = -\LL_{\vv} \rho ~,
    \qquad
    \delta \varsigma = -\LL_{\vv} \rho \varsigma ~,
    \qquad
    \delta \beta = -\LL_{\vv} \beta ~,
\end{equation}
where $\vv$ is a curve in $X(\Omega)$ which is null at both end-points.

The variational principle is supplemented with the following general advection equations : 
\begin{equation}
\partial_t \varrho + \LL_{\uu}\varrho  = 0 ~, 
\qquad 
\partial_t \varsigma + \LL_{\uu}\varsigma = 0 ~, 
\qquad 
\partial_t \beta + \LL_{\uu} \beta = 0 ~.
\end{equation}
We have here denoted $\LL_{\uu}$ the Lie derivative, defined by $\LL_{\uu} \alpha = \frac{d}{dt}|_{t=0} \phi^*_t \alpha$, where $\phi_t$ denotes the flow of $\uu$ and $\phi_t^*$ the pullback by $\phi_t$. We will here make extensive use of Cartan's ``magic'' formula $\LL_{\uu} \alpha = d i_{\uu} \alpha + i_{\uu} d \alpha$ 
where $i_{\uu}$ denotes the interior product.

\subsection{Vector fields as operators on forms}
\label{sec:Vector_as_operator}
When we look at the variational principle from the previous section, we see that vector fields appear mostly as acting on differential forms (except in the kinetic energy term). Indeed the action by Lie derivative of a vector field, gives a representation of $X(\Omega)$ in $\Lambda^0 \times \Lambda^1 \times ... \times \Lambda^n$ where $\Lambda^k$ denotes the differential k-forms on $\Omega$. Since this action commutes with the exterior derivative, we can identify $X(\Omega)$ as a subspace of the following Lie algebra : 
\begin{equation}
\label{eqn:g_deff}
\g = \{(A^0,...,A^n) \in L(\Lambda^0) \times ... \times L(\Lambda^n) : d A^{i-1} = A^i d \} ~ , 
\end{equation}
where the Lie bracket is given by the traditional commutator $([A,B])^i = A^i B^i - B^i A^i $ which clearly leaves $\g$ invariant.
This representation will be the center of the variational discretization we propose in the following section. 
Let denote $\g'$, the subalgebra of $\g$ of operators actually representing vector fields :
\begin{equation}
\label{eqn:g_prim_def}
\g' = \{A = (A^0,A^1,...,A^n) \in \g | \ \exists \uu \in X(\Omega) : A^i \alpha = \LL_{\uu} \alpha^i \ \forall \alpha^i \in \Lambda^i, i=0,...,n \} ~ .
\end{equation}
Then for any $A = (A^i)_{i=0,...,n} \in \g'$ there exists a unique vector field 
$\hat{A} \in X(\Omega)$ such that for all i-form $\alpha^i $,  the relation $\LL_{\hat{A}} \alpha^i = A^i \alpha^i$ holds.

We can therefore rewrite the previous Lagrangian as :
\begin{equation}
    \label{eqn:MHD_reduced_lagrangian_operator}
    l(A,\varrho, \varsigma, \beta)=\int_\Omega \frac{1}{2}\varrho|\hat{A}|^2-\varrho e(\varrho, \varsigma) -\frac{1}{2}\beta \wedge \star \beta ~,
\end{equation}
and the variational principle associated is $\delta S = 0$ (with $S$ the action of $l$ as defined in \cref{eqn:MHD_reduced_action}) under variations of the form
\begin{equation}
\delta A = \partial_t B + [B, A]~,
\quad
\delta \varrho = - B^n \varrho ~, 
\quad
\delta \varsigma = - B^n \varsigma ~,
\quad
\delta \beta = - B^{n-1} \beta ~,
\end{equation}
with $B = B(t)$ a general curve on $\g'$, null at end-points. This variational principle is to be supplemented with the following advection equations : 
\begin{equation}
\partial_t \varrho + A^n\varrho  = 0 ~, 
\quad
\partial_t \varsigma + A^n\varsigma = 0 ~, 
\quad
\partial_t \beta + A^{n-1} \beta = 0 ~.
\end{equation}
\section{Discretization}
\label{sec:discretization}

We now describe our discretization in a rather general setting. We first highlight the hypotheses on the discrete spaces and operators on which it relies, and later in \cref{sec:Xh_map_pi,sec:numerics} we will specify the choices used in our numerical experiments.

\subsection{Discrete forms : FEEC}
The preservation of some crucial invariants such as total mass, entropy and $\Div \BB = 0$ requires a good discrete representation of differential forms. What we primarily need is a discrete structure that preserves the de Rham sequence. In vector proxy language, we want discrete spaces on which the calculus identities $\curl \grad = \Div \curl =0$ hold. Such discretizations have been investigated during the last decades, giving birth to the theories of Discrete Exterior Calculus (DEC)~\cite{hirani2003discrete} and Finite Element Exterior Calculus (FEEC)~\cite{arnold2018finite,arnold_falk_winther_2006_anum}. This yields our first hypothesis :
\begin{hypothesis}
\label{hyp:DeRham}
We consider a sequence of discrete spaces $(V_h^i)_{i=0,...,n}$ such that $d^i(V_h^i) \subset V_h^{i+1}$.
This implies that the relation $d \circ d = 0$ holds in every discrete space.
\end{hypothesis}
In a 3 dimensional setting and with vector proxy notation, this means that we consider four spaces $(V^0_h, V^1_h, V^2_h, V^3_h)$, such that :
\begin{itemize}
\item The gradient is well defined on $V_h^0$ ($V_h^0 \subset H^1(\Omega)$) and $\grad(V_h^0) \subset V_h^1$.
\item The curl is well defined on $V_h^1$ ($V_h^1 \subset H(\curl,\Omega)$) and $\curl(V_h^1) \subset V_h^2$.
\item The divergence is well defined on $V_h^2$ ($V_h^2 \subset H(\Div,\Omega)$) and $\Div(V_h^2) \subset V_h^3$.
\end{itemize}
Hence, we will look for a discrete density $\rhoh \in V_h^n$, entropy $\sh \in V_h^n$ and magnetic field $\Bh \in V_h^{n-1}$.

\subsection{Discrete vector fields}

For the discretization of vector fields we follow the point of view presented in \cref{sec:Vector_as_operator}, where vector fields 
were seen as operators on differential forms, the identification being done through the Lie derivative. Mimicking the definition in \cref{eqn:g_deff} we propose a first discrete representation of the vector fields as : 
\begin{equation}
\label{eqn:gh_deff}
\g_h = \{(A^0_h,...,A_h^n) \in L(V^0_h) \times ... \times L(V^n_h) : d A_h^{i-1} = A_h^i d \} ~ . 
\end{equation}
One should note that this space is very big and doing computations with objects in $\g_h$ would make our method uncompetitive compared to state of the art finite element/DG methods. However, this space will not be used in the final scheme: it is introduced only for the 
theoretical purpose of mimicking \cref{eqn:g_deff} at a discrete level. 

The last object that is missing to provide a discretization of \cref{eqn:MHD_reduced_lagrangian_operator} is a $\hat{\cdot}$ operator which maps a Lie derivative to a corresponding vector field.
To describe our method in the most general setting, we consider a general discrete 
vector field space that we denote $X_h \subset X(\Omega)$. 
We first need to describe how, from a discrete vector field $\uh \in X_h$, 
we can recover an operator $A_{\uh}$ in $\g_h$. 
To do so we make another classical assumption.
\begin{hypothesis}
\label{hyp:Pi}
We suppose that there exist projection operators $\Pi^i$ onto the discrete 
spaces $V_h^i$, $i\in \{0,...,n\}$, that preserve the vanishing of forms at the boundary $\partial \Omega$.
\end{hypothesis}
We next use Cartan's formula and define for every $i\in \{0,...,n\}$ and every $\alpha_h \in V_h^i$, 
\begin{equation}
\label{eqn:discrete_advection}
A^i_{\uh} \alpha_h = d\Pi^{i-1}(i_{\uh} \alpha_h) + \Pi^{i}(i_{\uh} d \alpha_h) ~ ,
\end{equation}
where we remind that $i_{\uh}$ is the (continuous) interior product associated with a vector field $\uh$. 
Notice that the commutation property 
\begin{equation} 
    \label{eqn:comm_dA}
    d A^i_{\uh} \alpha_h = A^{i-1}_{\uh} d \alpha_h
\end{equation}
holds for all $\alpha_h \in V_h^i$, as a direct consequence of \cref{hyp:DeRham,hyp:Pi}.
In particular, \eqref{eqn:comm_dA} does not require projection operators $\Pi$ that commute with the exterior derivative.

This discretization of the Lie derivative operators is a central point in our approach. 
It will allow us to propose a strong formulation for all the advection equations, 
in contrast to previous variational derivations~\cite{natale2018variational,gawlik2021structure,gawlik2021variational} 
where some advection equations are discretized in a weak form.
We now make the following hypothesis in order to build a $\hat{\cdot}$ map :
\begin{hypothesis}
\label{hyp:hat-map}
We suppose that the map $\uu \mapsto A_{\uu}$ mapping $X_h$ to $\g_h$ is injective, and that we can build a pseudo-inverse. 
This left-inverse will be denoted as a hat map, that is :
\begin{equation}
\label{eqn:requirement_hat}
\hat{\cdot} ~ : ~ g_h \rightarrow X_h ~ \text{ such that } \widehat{A_{\uh}}=\uh ~ \forall \uh \in X_h ~ .
\end{equation}
\end{hypothesis}
\subsection{Discrete variational principle and scheme}

Using the discrete forms, vector fields and hat map described above, we define our discrete version 
of the Lagrangian in \cref{eqn:MHD_reduced_lagrangian_operator} as 
\begin{equation}
    \label{eqn:MHD_reduced_lagrangian_operator_discrete}
    l_h(A_h,\varrho_h, \varsigma_h, \beta_h)=\int_\Omega \frac{1}{2}\varrho_h|\hat{A_h}|^2-\varrho_h e(\varrho_h, \varsigma_h) -\frac{1}{2}\beta_h \wedge \star \beta_h ~,
\end{equation}
for $A_h \in \g_h$, $\varrho_h, \varsigma_h \in V^n_h$ and $\beta_h \in V^{n-1}_h$.
We next consider a first variational problem. 

\begin{definition}[General discrete variational problem]

Find $A_h \in \g_h$, $\varrho_h, \varsigma_h \in V^n_h$ and $\beta_h \in V^{n-1}_h$ that extremize the action of the 
discrete Lagrangian \eqref{eqn:MHD_reduced_lagrangian_operator_discrete}  
under variations of the form 
\begin{equation}
\delta A_h = \partial_t C_h + [C_h, A_h] ~, 
\qquad 
\delta \varrho_h = - C_h^n \varrho_h  ~, 
\qquad 
\delta \varsigma_h = - C_h^n \varsigma_h  ~, 
\qquad
\delta \beta_h = - C_h^{n-1} \beta_h ~,
\end{equation}
where $C_h = C_h(t)$ is a general curve on $\g_h$ which is null at end-points. This variational problem is supplemented with the following advection equations : 
\begin{equation}
\partial_t \varrho_h + A_h^n\varrho_h  = 0 ~, 
\qquad
\partial_t \varsigma_h + A_h^n\varsigma_h = 0 ~, 
\qquad
\partial_t \beta_h + A_h^{n-1} \beta_h = 0 ~.
\end{equation}
\end{definition}
The equation resulting from this variational principle is:
\begin{equation}
\langle \partial_t \frac{\delta l_h}{\delta A_h}, C_h \rangle - \langle \frac{\delta l_h}{\delta A_h}, [C_h,A_h] \rangle + \langle \frac{\delta l_h}{\delta \varrho_h}, C_h^n \varrho_h \rangle + \langle \frac{\delta l_h}{\delta \varsigma_h}, C_h^n \varsigma_h \rangle + \langle \frac{\delta l_h}{\delta \beta_h}, C_h^{n-1} \beta_h \rangle = 0 
\end{equation}
for all $C_h \in \g_h$. We have $\langle \frac{\delta l_h}{\delta A_h}, C_h \rangle = \int_{\Omega} \varrho_h \hat{A_h} \cdot \hat{C_h}$ hence, for this equation to have a unique solution, one needs 
that the hat map be not only surjective but also injective. 
Therefore, we need to reduce the variational principle to a smaller space:
a simple choice is to use the image of $\uu_h \mapsto A_{\uu_h}$.
Thus, we set 
\begin{equation}
    \Delta_h = \{A_{\uu_h} \in \g_h : \uu_h \in X_h \}    
\end{equation}
and consider the restricted variational problem 
where $A_h$ is sought for in $\Delta_h$, with variations $C_h$ also in $\Delta_h$. This leads us to the following problem.

\begin{definition}[Restricted discrete variational problem]
\label{def:disc_var_prob}
Find $A_h \in \Delta_h$, $\varrho_h, \varsigma_h \in V^n_h$ and $\beta_h \in V^{n-1}_h$ that extremizes the action of the Lagrangian \eqref{eqn:MHD_reduced_lagrangian_operator_discrete} under variations of the form 
\begin{equation}
\delta A_h = \partial_t C_h + [C_h, A_h] ~, 
\qquad 
\delta \varrho_h = - C_h^n \varrho_h  ~, 
\qquad 
\delta \varsigma_h = - C_h^n \varsigma_h  ~, 
\qquad
\delta \beta_h = - C_h^{n-1} \beta_h ~,
\end{equation}
where $C_h = C_h(t)$ is a general curve on $\Delta_h$ which is null at end-points, together with the advection equations:  
\begin{equation}
\partial_t \varrho_h + A_h^n\varrho_h  = 0 ~, 
\qquad
\partial_t \varsigma_h + A_h^n\varsigma_h = 0 ~, 
\qquad
\partial_t \beta_h + A_h^{n-1} \beta_h = 0 ~.
\end{equation}
\end{definition}
\begin{remark}
The space $\Delta_h$ is in general not stable by the bracket operation. 
The restriction to a space not stable by the bracket is called a 
``non-holonomic constraint'', see~\cite{gawlik2011geometric, pavlov2011structure, gawlik2021variational}.
\end{remark}

\begin{proposition}
\label{prop:explicit_momentum}
Let $A_h \in \Delta_h$ be the discrete Lie derivative operator 
in the solution to the Problem given in Def.~\ref{def:disc_var_prob}. 
Then $A_h = A_{\uu_h}$ for some $\uu_h \in X_h$ which satisfies :
\begin{equation}
\label{eqn:momentum_eq_Xh_1}
\begin{aligned}
& \int_\Omega \partial_t (\varrho_h \uu_h) \cdot \vv_h - (\varrho_h \uu_h) \cdot \widehat{[A_{\vv_h},A_{\uu_h}]}  
  + \Big(\frac{1}{2}|\uu_h|^2 - e(\varrho_h, \varsigma_h) - \varrho_h \partial_{\varrho_h} e(\varrho_h, \varsigma_h) \Big) A_{\vv_h}^{n} \varrho_h 
  \\
  	&\mspace{260mu} 
  - \varrho_h \partial_{\varsigma_h}e(\varrho_h, \varsigma_h)  A_{\vv_h}^{n} \varsigma_h - \beta_h \wedge \star A_{\vv_h}^{n-1} \beta_h = 0 ~ \qquad \forall \vv_h \in X_h ~.
\end{aligned}
\end{equation}
\end{proposition}

\begin{proof}
Writing that our solution is an extremum of the action of 
\eqref{eqn:MHD_reduced_lagrangian_operator_discrete} is equivalent to 
\begin{equation*}
\langle \partial_t \frac{\delta l_h}{\delta A_h}, C_h \rangle - \langle \frac{\delta l_h}{\delta A_h}, [C_h,A_h] \rangle + \langle \frac{\delta l_h}{\delta \varrho_h}, C_h^{n} \varrho_h \rangle + \langle \frac{\delta l_h}{\delta \varsigma_h}, C_h^n \varsigma_h \rangle + \langle \frac{\delta l_h}{\delta \beta_h}, C_h^{n-1} \beta_h \rangle = 0 
~ \quad \forall C_h \in \Delta_h ~.
\end{equation*}
We can now rewrite this equation by expliciting all the terms : 
\begin{align*}
&\int_\Omega \partial_t (\varrho_h \hat{A_h}) \cdot \hat{C_h} - (\varrho_h \hat{A_h}) \cdot \widehat{[C_h,A_h]}  + \Big(\frac{1}{2}|\hat{A_h}|^2 - e(\varrho_h, \varsigma_h) - \varrho_h \partial_{\varrho_h} e(\varrho_h, \varsigma_h) \Big) C_h^{n} \varrho_h   \\
&\mspace{260mu}	 
-\varrho_h \partial_{\varsigma_h}e(\varrho_h, \varsigma_h)  C_h^{n} \varsigma_h - \beta_h \wedge \star C_h^{n-1} \beta_h  = 0 ~ \qquad \forall C_h \in \Delta_h ~.
\end{align*}
As all the operators in $\Delta_h$ are images of vector fields in $X_h$, $A_h = A_{\uu_h}$ for some $\uu_h \in X_h$ and $C_h = A_{\vv_h}$ for some $\vv_h \in X_h$. Using the relation $\widehat{A_{\uh}}=\uh$ we can rewrite this equation on $\uu_h$ and $\vv_h$ and directly find \cref{eqn:momentum_eq_Xh_1}
\end{proof}

Putting together this momentum equation and the advection equations, we arrive at the following result 
which defines our variational scheme.

\begin{proposition}[semi-discrete scheme]

Solutions to Problem \cref{def:disc_var_prob} satisfy the following equations 
(where we have identified $\widehat{A_{\uu_h}}$ with $\uu_h$) :
\begin{subequations}
\label{eqn:discrete_Xh}
\begin{equation}
\label{eqn:momentum_eq_Xh}
\begin{aligned}
&\int_\Omega \partial_t (\varrho_h \uu_h) \cdot \vv_h - (\varrho_h \uu_h) \cdot \widehat{[A_{\vv_h},A_{\uu_h}]}  + \Big(\frac{1}{2}|\uu_h|^2 - e(\varrho_h, \varsigma_h) - \varrho_h \partial_{\varrho_h} e(\varrho_h, \varsigma_h) \Big) A_{\vv_h}^{n} \varrho_h \\
&\mspace{260mu}	 
-	 \varrho_h \partial_{\varsigma_h}e(\varrho_h, \varsigma_h)  A_{\vv_h}^{n} \varsigma_h 
  - \beta_h \wedge \star A_{\vv_h}^{n-1} \beta_h  = 0 ~ \qquad \forall \vv_h \in X_h ~,
\end{aligned}
\end{equation}
\begin{equation}
\label{eqn:advection_eq_Xh_rho}
\partial_t \varrho_h + A_{\uu_h}^n\varrho_h  = 0 ~, 
\end{equation}
\begin{equation}
\label{eqn:advection_eq_Xh_s}
\partial_t \varsigma_h + A_{\uu_h}^n\varsigma_h = 0 ~,
\end{equation}
\begin{equation}
\label{eqn:advection_eq_Xh_B}
\partial_t \beta_h + A_{\uu_h}^{n-1} \beta_h = 0 ~.
\end{equation}
\end{subequations}
These equations constitute the semi-discrete scheme that we consider in the remaining of this article.
\end{proposition}
\begin{proof}
The momentum equation is \cref{eqn:momentum_eq_Xh_1}: it was given in \cref{prop:explicit_momentum}. 
The advection equations follow directly from \cref{def:disc_var_prob}, using that $A_h = A_{\uu_h}$.
\end{proof}

Again let us point out that, as a result of our strong discretization of the Lie advection operators, only the momentum equation is weak here. In particular, all the discrete transport equations hold in a strong sense. 

Later we will specify one choice of discretization spaces and operators that is used in the numerical results presented in this article. 
In particular we refer the reader to \cref{sec:Xh_map_pi} for the choice of the projection $\Pi$ and the hat map $\hat{\cdot}$, to \cref{sec:FEM_proxy} for a reformulation of the scheme in more standard terms and to \cref{sec:discrete_Vh} for a particular choice of discrete FEEC spaces.
But first, we establish some conservation laws that our discrete framework guarantees.

\subsection{Conservation properties of the general (semi-discrete) scheme}

\begin{proposition}[Conservation of mass]
\label{prop:mass_conv}
Discrete solutions to \cref{eqn:discrete_Xh}, under \cref{hyp:DeRham,hyp:Pi,hyp:hat-map}, preserve the total mass $\int_\Omega \varrho_h$.
\end{proposition}
\begin{proof}
We use the advection 
\cref{eqn:advection_eq_Xh_rho} for $\varrho_h$, noting that 
$d\varrho_h = 0$ for an $n$-form: this gives
$$
\frac{\rmd}{\rmd t} \int_\Omega \varrho_h = \int_\Omega \partial_t \varrho_h 
= -\int_\Omega A^n_{\uu_h} \varrho_h 
= -\int_\Omega d(\Pi^{n-1} (i_{\uu_h} \varrho_h)) 
= -\int_{\partial \Omega} \Pi^{n-1} (i_{\uu_h} \varrho_h) = 0
$$
where we haved used that $\uu_h$ is tangent to the boundary, thus the flux of $i_{\uu_h} \varrho_h$ vanishes
on $\partial \Omega$, which is preserved by $\Pi^{n-1}$ by assumption.
\end{proof}

\begin{proposition}[Conservation of entropy]
\label{prop:entrop_conv}
Discrete solutions to \cref{eqn:discrete_Xh}, under \cref{hyp:DeRham,hyp:Pi,hyp:hat-map}, preserve the total entropy $\int_\Omega \varsigma_h$.
\end{proposition}
\begin{proof}
The proof is similar to the previous one on mass preservation, using \cref{eqn:advection_eq_Xh_s}.
\end{proof}

\begin{proposition}[Preservation of solenoidal character of $\BB_h$]
\label{prop:divB_conv}
In addition to \cref{hyp:DeRham,hyp:Pi,hyp:hat-map}, assume that $d \beta_h(t=0) = 0$. 
Then discrete solutions to \cref{eqn:discrete_Xh} satisfy $d \beta_h = 0$ (that is, $\Div \BB_h = 0$) for all $t$.  
\end{proposition}
\begin{proof}
Assuming $d\beta_h(t) = 0$, we show that $\partial_t d \beta_h(t) = 0$, which implies that $d \beta_h =0 $ at all $t$: we have
\begin{equation*}
    \partial_t d \beta_h = d \partial_t \beta_h
    = d(d \Pi^{n-2}(i_{\uu_h} \beta_h) + \Pi^{n-1}(i_{\uu_h}(d\beta_h))
\end{equation*}
where we have used \cref{eqn:advection_eq_Xh_B} and \cref{eqn:discrete_advection}
The first term vanishes because $d \circ d=0$ and the second one because we assumed $d \beta_h(t) =0$.
\end{proof}

\begin{proposition}[Preservation of energy]
Discrete solutions to \cref{eqn:discrete_Xh}, under \cref{hyp:DeRham,hyp:Pi,hyp:hat-map}, preserve the energy 
\begin{equation}
\label{eqn:energy}
E = \int_\Omega \varrho_h \frac{|\uu_h|^2}{2} + \varrho_h e(\varrho_h, \varsigma_h) + \frac{|\beta_h|^2}{2}
\end{equation}
where we have denoted $|\beta|^2 = \beta_h \wedge \star \beta_h$.
\end{proposition}
\begin{proof}
First we compute : 
\begin{align*}
\partial_t (\varrho_h \frac{|\uu_h|^2}{2}) &= \frac{1}{2}\partial_t(\varrho_h \uu_h) \cdot \uu_h + \frac{1}{2}(\varrho_h \uu_h) \partial_t \uu_h \\
&= \partial_t(\varrho_h \uu_h) \cdot \uu_h-\frac{1}{2}\partial_t(\varrho_h \uu_h) \cdot \uu_h + \frac{1}{2}(\varrho_h \uu_h) \partial_t \uu_h \\
&= \partial_t(\varrho_h \uu_h) \cdot \uu_h-\frac{1}{2}\partial_t(\varrho_h) \cdot |\uu_h|^2 - \frac{1}{2}(\varrho \uu_h) \partial_t \uu_h + \frac{1}{2}(\varrho \uu_h) \partial_t \uu_h \\
&= \partial_t(\varrho_h \uu_h) \cdot \uu_h-\frac{1}{2}\partial_t(\varrho_h) \cdot |\uu_h|^2 ~.
\end{align*}
Adding further the other terms, we have
\begin{align*}
\partial_t E &= \int_\Omega \partial_t(\varrho_h \uu_h) \cdot \uu_h-\frac{1}{2}\partial_t(\varrho_h) \cdot |\uu_h|^2 +  \Big( e(\varrho_h, \varsigma_h) + \varrho_h \partial_{\varrho_h} e(\varrho_h, \varsigma_h) \Big) \partial_t \varrho_h + \varrho_h \partial_{\varsigma_h}e(\varrho_h, \varsigma_h)  \partial_t \varsigma_h + \beta_h \wedge \star \partial_t \beta_h \\
&= \int_\Omega \partial_t(\varrho_h \uu_h) \cdot \uu_h +\Big(\frac{1}{2}\partial_t(\varrho_h) \cdot |\uu_h|^2 - e(\varrho_h, \varsigma_h) - \varrho_h \partial_{\varrho_h} e(\varrho_h, \varsigma_h) \Big) A^n_{\uu_h} \varrho_h - \varrho_h \partial_{\varsigma_h}e(\varrho_h, \varsigma_h) A^n_{\uu_h} \varsigma_h - \beta_h \wedge \star A^{n-1}_{\uu_h} \beta_h ~,
\end{align*}
where we have used the advection equations \cref{eqn:advection_eq_Xh_rho,eqn:advection_eq_Xh_s,eqn:advection_eq_Xh_B}.
Using next that $[A_{\uu_h}, A_{\uu_h}] = 0$, this is equal to 
\begin{align*}
\partial_t E &= \int_\Omega \partial_t(\varrho_h \uu_h) \cdot \uu_h - \partial_t(\varrho_h \uu_h) \cdot \widehat{[A_{\uu_h}, A_{\uu_h}]} +\Big(\frac{1}{2}\partial_t(\varrho_h) \cdot |\uu_h|^2 - e(\varrho_h, \varsigma_h) - \varrho_h \partial_{\varrho_h} e(\varrho_h, \varsigma_h) \Big) A^n_{\uu_h} \varrho_h \\
    & \mspace{200mu} 
    - \varrho_h \partial_{\varsigma_h}e(\varrho_h, \varsigma_h) A^n_{\uu_h} \varsigma_h - \beta_h \wedge \star A^{n-1}_{\uu_h} \beta_h 
    \\&=0 ~,
\end{align*}
using \cref{eqn:momentum_eq_Xh} with $\vv_h = \uu_h$: this completes the proof.
\end{proof}

\subsection{Time discretization}
\label{sec:time_discretization}
We now propose a time discretization of \cref{eqn:discrete_Xh} that preserves at the fully discrete level the invariants 
mentioned in the previous section. Here we suppose that the time step $\Delta t$ is constant, 
and for any discrete variable we denote $a^k = a(\cdot, k \Delta t)$ and $a^{k+\frac{1}{2}}=\frac{a^{k+1}+a^k}{2}$ for conciseness.
Again, we assume that \cref{hyp:DeRham,hyp:Pi,hyp:hat-map} hold throughout this section.

Our time discretization reads :
\begin{subequations}
\label{eqn:FEMt_full}
\begin{multline}
        \label{eqn:FEMt_full_MHD_mom}
    \int_\Omega  \frac{\varrho_h^{k+1} \uu_h^{k+1} - \varrho_h^k \uu_h^k}{\Delta t} \cdot \vv_h - \varrho_h^{k + \frac{1}{2}} \uu^{k + \frac{1}{2}}_h \cdot \widehat{[A_{\vv_h},A_{\uu_h^{k+\frac{1}{2}}}]}
    \\
    + (\frac{\uu_h^k \cdot \uu_h^{k+1}}{2} - \frac{1}{2}(\frac{\varrho_h^{k+1}e(\varrho_h^{k+1},\varsigma_h^{k+1})-\varrho_h^{k}e(\varrho_h^{k},\varsigma_h^{k+1})}{\varrho_h^{k+1}-\varrho_h^k} + \frac{\varrho_h^{k+1}e(\varrho_h^{k+1},\varsigma_h^{k})-\varrho_h^{k}e(\varrho_h^{k},\varsigma_h^{k})}{\varrho_h^{k+1}-\varrho_h^k})) A_{\vv_h} \varrho_h^{k+\frac{1}{2}} \\
    - \frac{1}{2}(\frac{\varrho_h^{k+1}e(\varrho_h^{k+1},\varsigma_h^{k+1})-\varrho_h^{k+1}e(\varrho_h^{k+1},\varsigma_h^{k})}{\varsigma_h^{k+1}-\varsigma_h^k}+\frac{\varrho_h^{k}e(\varrho_h^{k},\varsigma_h^{k+1})-\varrho_h^{k}e(\varrho_h^{k},\varsigma_h^{k})}{\varsigma_h^{k+1}-\varsigma_h^k}) A_{\vv_h} \varsigma_h^{k+\frac{1}{2}} \\
    - \beta_h^{k+\frac{1}{2}} \wedge \star A^{n-1}_{\vv_h} \beta_h^{k+\frac{1}{2}} = 0 ~ \quad \forall \vv_h \in X_h ~,  
\end{multline}
\begin{equation}
    \label{eqn:FEMt_full_MHD_dens}
    \frac{\varrho_h^{k+1} - \varrho_h^k }{\Delta t} + A_{\uu_h^{k+\frac{1}{2}}}^n \varrho_h^{k+\frac{1}{2}} = 0 ~,
\end{equation}
\begin{equation}
    \label{eqn:FEMt_full_MHD_entrop}
    \frac{\varsigma_h^{k+1} - \varsigma_h^k }{\Delta t} + A_{\uu_h^{k+\frac{1}{2}}}^n \varsigma_h^{k+\frac{1}{2}} = 0 ~,
\end{equation}
\begin{equation}
    \label{eqn:FEMt_full_MHD_B}
    \frac{\beta_h^{k+1} - \beta_h^k }{\Delta t} + A_{\uu_h^{k+\frac{1}{2}}}^{n-1} \beta^{k+\frac{1}{2}} = 0 ~.
\end{equation}
\end{subequations}
\begin{proposition}
The scheme \eqref{eqn:FEMt_full} preserves the energy 
$E^k = \int_\Omega \frac{1}{2}\varrho^k_h|\uu_h^k|^2+\varrho^k_h e(\varrho^k_h,\varsigma^k_h)+\frac{1}{2}|\beta^k_h|^2$.
\end{proposition}
\begin{proof}

We test \cref{eqn:FEMt_full_MHD_mom} against $\vv_h = \uu_h^{k+\frac{1}{2}}$ and develop this term by term : The first term is
\begin{align*}
        \int_\Omega  \frac{\varrho_h^{k+1} \uu_h^{k+1} - \varrho_h^k \uu_h^k}{\Delta t} \cdot \uu_h^{k+\frac{1}{2}} &= \frac{1}{\Delta t}\int_\Omega \varrho_h^{k+1}\frac{\uu_h^{k+1}\cdot \uu_h^{k+1}}{2}-\varrho_h^{k}\frac{\uu_h^{k}\cdot \uu_h^{k}}{2}+(\varrho_h^{k+1}-\varrho_h^k)\frac{\uu_h^{k+1}\cdot \uu_h^k}{2} \\
        &= \frac{1}{\Delta t}(E^{k+1}_\kin-E^k_\kin)+ \int_\Omega \frac{\varrho_h^{k+1}-\varrho_h^k}{\Delta t}\frac{\uu_h^{k+1}\cdot \uu_h^k}{2} 
\end{align*}
where we have denoted by $E^k_\kin = \int_\Omega \frac 12 \varrho_h^k |\uu_h^k|^2$ the kinetic energy.
The second term $\varrho_h^{k + \frac{1}{2}} \uu^{k + \frac{1}{2},i}_h \cdot \widehat{[A_{\uu_h^{k+\frac{1}{2}}},A_{\uu_h^{k+\frac{1}{2}}}]}$ is zero by skew symmetry.   
The next terms read
\begin{align*}
    \int_\Omega \frac{\uu_h^k \cdot \uu_h^{k+1}}{2} A^{n}_{\uu_h^{k+\frac{1}{2}}} \varrho_h^{k+\frac{1}{2}}
    = -\int_\Omega\frac{\uu_h^k \cdot \uu_h^{k+1}}{2} \frac{\varrho_h^{k+1} - \varrho_h^k }{\Delta t}
\end{align*}
\begin{align*}
    \int_\Omega - \frac{\varrho_h^{k+1}e(\varrho_h^{k+1},\varsigma_h^{k+1})-\varrho_h^{k}e(\varrho_h^{k},\varsigma_h^{k+1})}{\varrho_h^{k+1}-\varrho_h^k} A^{n}_{\uu_h^{k+\frac{1}{2}}}\varrho_h^{k+\frac{1}{2}}
    &= \int_\Omega \frac{\varrho_h^{k+1}e(\varrho_h^{k+1},\varsigma_h^{k+1})-\varrho_h^{k}e(\varrho_h^{k},\varsigma_h^{k+1})}{\varrho_h^{k+1}-\varrho_h^k} \frac{\varrho_h^{k+1} - \varrho_h^k }{\Delta t} \\
    &= \int_\Omega\frac{\varrho_h^{k+1}e(\varrho_h^{k+1},\varsigma^{k+1})-\varrho_h^{k}e(\varrho_h^{k},\varsigma_h^{k+1})}{\Delta t}  
\end{align*}
\begin{align*}
    \int_\Omega - \frac{\varrho_h^{k+1}e(\varrho_h^{k+1},\varsigma_h^{k})-\varrho_h^{k}e(\varrho_h^{k},\varsigma_h^{k})}{\varrho_h^{k+1}-\varrho_h^k} A^{n}_{\uu_h^{k+\frac{1}{2}}}\varrho_h^{k+\frac{1}{2}} 
    = \int_\Omega\frac{\varrho_h^{k+1}e(\varrho_h^{k+1},\varsigma_h^{k})-\varrho_h^{k}e(\varrho_h^{k},\varsigma_h^{k})}{\Delta t}  
\end{align*}
\begin{align*}
    \int_\Omega - \frac{\varrho_h^{k+1}e(\varrho_h^{k+1},\varsigma_h^{k+1})-\varrho_h^{k+1}e(\varrho_h^{k+1},\varsigma_h^{k})}{\varsigma_h^{k+1}-\varsigma_h^k}A^{n}_{\uu_h^{k+\frac{1}{2}}} \varsigma_h^{k+\frac{1}{2}}
    = \int_\Omega\frac{\varrho_h^{k+1}e(\varrho_h^{k+1},\varsigma_h^{k+1})-\varrho_h^{k+1}e(\varrho_h^{k+1},\varsigma_h^{k})}{\Delta t}  
\end{align*}
\begin{align*}
    \int_\Omega - \frac{\varrho_h^{k}e(\varrho_h^{k},\varsigma_h^{k+1})-\varrho_h^{k}e(\varrho_h^{k},\varsigma_h^{k})}{\varsigma_h^{k+1}-\varsigma_h^k} A^{n}_{\uu_h^{k+\frac{1}{2}}} \varsigma_h^{k+\frac{1}{2}}
    = \int_\Omega\frac{\varrho_h^{k}e(\varrho_h^{k},\varsigma_h^{k+1})-\varrho_h^{k}e(\varrho_h^{k},\varsigma_h^{k})}{\Delta t}  
\end{align*}
\begin{align*}
    \int_\Omega \beta_h^{k+\frac{1}{2}} \wedge \star A^{n-1}_{\uu_h^{k+\frac{1}{2}}}\beta_h^{k+\frac{1}{2}} = \int_\Omega \beta_h^{k+\frac{1}{2}} \wedge \star \frac{\beta_h^{k+1} - \beta_h^k }{\Delta t} = \frac{1}{\Delta t} \int_\Omega \frac{\beta_h^{k+1}\wedge \star \beta_h^{k+1}}{2}-\frac{\beta_h^{k}\wedge \star \beta_h^{k}}{2}~.
\end{align*}
Summing all these terms gives
\begin{align*}
    \frac{1}{\Delta t}(E^{k+1}_\kin-E^k_\kin)+\int_\Omega\frac{\varrho_h^{k+1}e(\varrho_h^{k+1},\varsigma_h^{k+1})-\varrho_h^{k}e(\varrho_h^    {k},\varsigma^{k})}{\Delta t} + \frac{1}{\Delta t} \int_\Omega \frac{\beta_h^{k+1}\wedge \star \beta_h^{k+1}}{2}-\frac{\beta_h^{k}\wedge \star \beta_h^{k}}{2} =0
\end{align*}
so that the total energy is indeed preserved.
\end{proof}

\begin{proposition}
The fully discrete scheme \cref{eqn:FEMt_full} preserves the total mass, entropy and the property $d\beta_h = 0$
corresponding to the solenoidal (divergence-free) character of $\BB_h$.
\end{proposition}
\begin{proof}
The proof is a direct adaptation of the semi-discrete case, namely \cref{prop:mass_conv,prop:entrop_conv,prop:divB_conv}.
In particular we note that these conservation properties are not specific to this time discretization.
\end{proof}

\begin{proposition}
The fully discrete scheme \cref{eqn:FEMt_full} is reversible in time. 
Specifically, the solution map 
$S(\Delta t): (\uu_h^k, \varrho_h^k, \varsigma_h^k, \beta_h^k) \mapsto (\uu_h^{k+1}, \varrho_h^{k+1}, \varsigma_h^{k+1}, \beta_h^{k+1})$ 
defined by \cref{eqn:FEMt_full_MHD_mom,eqn:FEMt_full_MHD_dens,eqn:FEMt_full_MHD_entrop,eqn:FEMt_full_MHD_B} 
satisfies
$$
S(-\Delta t) = S(\Delta t)^{-1}.
$$
\end{proposition}
\begin{proof}
To show this property we need to verify that for a reversed time step $\widetilde {\Delta t} = -\Delta t$,
the solution $(\tilde \uu_h^{k+1}, \tilde \varrho_h^{k+1}, \tilde \varsigma_h^{k+1}, \tilde \beta_h^{k+1})$ 
starting from 
$(\tilde \uu_h^{k}, \tilde \varrho_h^{k}, \tilde \varsigma_h^{k}, \tilde \beta_h^{k}) = (\uu_h^{k+1}, \varrho_h^{k+1}, \varsigma_h^{k+1}, \beta_h^{k+1})$ coincides with $(\uu_h^k, \varrho_h^k, \varsigma_h^k, \beta_h^k)$.
This comes from direct computation, switching the indices $k$ and $k+1$
in \cref{eqn:FEMt_full_MHD_mom,eqn:FEMt_full_MHD_dens,eqn:FEMt_full_MHD_entrop,eqn:FEMt_full_MHD_B}: 
indeed all the terms whose sign is changed are multiplied by $\Delta t$, 
and $(\uu_h^{k+\frac{1}{2}}, \varrho_h^{k+\frac{1}{2}}, \varsigma_h^{k+\frac{1}{2}}, \beta_h^{k+\frac{1}{2}})$ 
is left unchanged 
(note that the discrete gradient terms have a change of sign in both the numerator and denominator).
\end{proof}

\subsection{Particular choice for $X_h$, hat map and projection $\Pi$}
\label{sec:Xh_map_pi}

So far, we have only presented our method in the setting of general discrete de Rham sequences and operators 
satisfying \cref{hyp:DeRham,hyp:Pi,hyp:hat-map}. 
Before presenting our numerical results, let us now specify some of the choices that will be used in our implementation. 

Our first choice is to use $X_h = (V^0_h)^n$, which means that we represent vector fields as vectors of scalar functions
(from the same discrete space as our 0-forms). 
The hat map will then be given by 
\begin{equation}
\label{eqn:hat_map}
\hat{A}_i = A^0(x_i) ~ ,
\end{equation}
where $x_i$ denotes the $i$-th coordinate seen as a function (0-form), 
and where $\hat{A}_i$ is the $i$-th component of the vector-valued function $\hat{A} \in X_h$. This definition mimics the formula 
\begin{equation}
\label{eqn:justif_def_hat}
u_i = \uu \cdot \grad (x_i) ~ ,
\end{equation}
valid for any vector field $\uu$ with components $u_i$. We note that this particular choice is used in other variational discretization, see e.g.~\cite{pavlov2011structure,natale2018variational,gawlik2021variational}.
An important observation is that this hat map provides a left-inverse for the 
discrete Lie derivative \eqref{eqn:discrete_advection}, independently of the underlying projection operator $\Pi$.

\begin{proposition}

The hat map $\hat{\cdot}$ defined by \cref{eqn:hat_map} satisfies \cref{hyp:hat-map} and gives the following formula for the bracket term : 
\begin{equation}
\widehat{[A_{\vv},A_{\uu}]_i} = \Pi^0(\vv \cdot \grad u_i - \uu \cdot \grad v_i) ~\quad ~ \forall \uu, \vv \in X_h~.
\end{equation}
\end{proposition}

\begin{proof}
We first check that this definition complies with \cref{hyp:hat-map}, that is the hat-map is a left inverse for $\uu \mapsto A_{\uu}$ : for any $\uu \in X_h$ we have indeed
\begin{equation*}
(\widehat{A_{\uu}})_i = A_{\uu}^0(x_i) = \Pi^0 (i_{\uu} d x_i) = \Pi^0 (i_{\uu} \ee_i) = \Pi^0 (u_i) = u_i
\end{equation*}
where the first equality is \cref{eqn:hat_map}, the second one is \eqref{eqn:discrete_advection} 
(noting that $i_{\uu}$ vanishes on any 0-form), the third one uses that 
$d x_i = \grad x_i = \ee_i$ the $i$-th canonical basis vector, and the last one uses that $\Pi^0$ is a projection onto $V^0_h$.
As for the bracket formula, we write for $\uu$ and $\vv$ in $X_h$
\begin{align*}
\widehat{[A_{\vv},A_{\uu}]_i} & = [A_{\vv},A_{\uu}](x_i) \\
	& = A_{\vv} u_i - A_{\uu} v_i \\
	& = \Pi^0(i_{\vv} d u_i - i_{\uu} d v_i)\\
	& = \Pi^0(\vv \cdot \grad u_i - \uu \cdot \grad v_i) ~.
\end{align*}
\end{proof}

Finally, for our numerical study we use commuting projections based on interpolation/histopolation for $\Pi$, which provide good results and simple implementation for our test cases in cartesian domains. 
Some tests were conducted with the $L^2$ projection but no substantial difference was observed.

\subsection{FEM equation in vector proxy notation}
\label{sec:FEM_proxy}
For completeness we reformulate our schemes in a more standard vector notation.
The semi-discrete scheme \cref{eqn:discrete_Xh} reads : find $\uu_h \in X_h$, $\rho_h, s_h \in V^n_h$ and $\BB_h \in V^{n-1}_h$ such that
\begin{subequations}
\label{eqn:eq_semi_proxy}
\begin{equation}
\begin{aligned}
&\int_\Omega \partial_t (\rho_h \uu_h) \cdot \vv_h 
  - \rho_h \sum_{i=1}^n u_i \big(\Pi^0(\vv_h \cdot \grad u_i - \uu_h \cdot \grad v_i)\big) 
  \\
  & \mspace{100mu}+ \Big(\frac{1}{2}|\uu_h|^2 - e(\rho_h, s_h) - \rho_h \partial_{\rho_h} e(\rho_h, s_h) \Big) \Div\Pi^2(\rho_h \vv_h) \\
  & \mspace{180mu}
  - 	 \rho_h \partial_{s_h}e(\rho_h, s_h)  \Div\Pi^2(s_h \vv_h) - \BB_h \cdot \curl \Pi^1(\BB_h \times \vv_h) = 0 \qquad
  \forall \vv_h \in X_h
\end{aligned}
\end{equation}
\begin{equation}
\partial_t \rho_h + \Div\Pi^2(\rho_h \uu_h)  = 0 ~, 
\end{equation}
\begin{equation}
\partial_t s_h + \Div\Pi^2(s_h \uu_h) = 0 ~, 
\end{equation}
\begin{equation}
\partial_t \BB_h + \curl \Pi^1(\BB_h \times \uu_h) = 0 ~.
\end{equation}
\end{subequations}
The fully discrete scheme \cref{eqn:FEMt_full} reads:
\begin{subequations}
\label{eqn:eq_full_proxy}
\begin{multline} 
    \int_\Omega  \frac{\rho_h^{k+1} \uu_h^{k+1} - \rho_h^k \uu_h^k}{\Delta t} \cdot \vv_h 
    - \rho_h^{k + \frac{1}{2}} \sum_{i=1}^n u^{k + \frac{1}{2}}_i \big(\Pi^0(\vv \cdot \grad u_i^{k + \frac{1}{2}} - \uu^{k + \frac{1}{2}} \cdot \grad v_i) \big)
    \\
    + (\frac{\uu_h^k \cdot \uu_h^{k+1}}{2} - \frac{1}{2}(\frac{\rho_h^{k+1}e(\rho_h^{k+1},s_h^{k+1})-\rho_h^{k}e(\rho_h^{k},s_h^{k+1})}{\rho_h^{k+1}-\rho_h^k} + \frac{\rho_h^{k+1}e(\rho_h^{k+1},s_h^{k})-\rho_h^{k}e(\rho_h^{k},s_h^{k})}{\rho_h^{k+1}-\rho_h^k})) \Div\Pi^2(\rho_h^{k+\frac{1}{2}} \vv_h)  \\
    - \frac{1}{2}(\frac{\rho_h^{k+1}e(\rho_h^{k+1},s_h^{k+1})-\rho_h^{k+1}e(\rho_h^{k+1},s_h^{k})}{s_h^{k+1}- s_h^k}+\frac{\rho_h^{k}e(\rho_h^{k},s_h^{k+1})-\rho_h^{k}e(\rho_h^{k},s_h^{k})}{s_h^{k+1}-s_h^k}) \Div\Pi^2(s_h^{k+\frac{1}{2}} \vv_h) \\
    - \BB_h^{k+\frac{1}{2}} \cdot \curl \Pi^1 (\BB_h^{k+\frac{1}{2}} \times \vv_h) = 0 ~ \qquad \forall \vv_h \in X_h ~,  
\end{multline}
\begin{equation}
    \frac{\rho_h^{k+1} - \rho_h^k }{\Delta t} + \Div\Pi^2(\rho_h^{k+\frac{1}{2}} \uu_h^{k+\frac{1}{2}}) = 0 ~, 
\end{equation}
\begin{equation}
    \frac{s_h^{k+1} - s_h^k }{\Delta t} + \Div\Pi^2(s_h^{k+\frac{1}{2}} \uu_h^{k+\frac{1}{2}}) = 0 ~, 
\end{equation}
\begin{equation}
    \frac{\BB_h^{k+1} - \BB_h^k }{\Delta t} + \curl \Pi^1 (\BB_h^{k+\frac{1}{2}} \times \uu_h^{k+\frac{1}{2}}) = 0 ~.
\end{equation}
\end{subequations}
\section{Numerical results}
\label{sec:numerics}

\subsection{Choice of discrete spaces} 
\label{sec:discrete_Vh}

In this section we study the accuracy of our 
scheme \cref{eqn:eq_full_proxy}, as well as 
the numerical conservation of the invariants claimed in \cref{sec:time_discretization}.
These numerical experiments were obtained with an implementation using 
the Psydac Finite Element library~\cite{guclu2022psydac}, 
which involves isogeometric tensor-product spline finite elements that fit in the FEEC framework.
These spaces are built via 1D spline spaces of maximal regularity, with anisotropic polynomial degree for the discrete $H(\Div)$ space to preserve the de Rham structure at the discrete level. More precisely, we consider the following discrete de Rham sequence 
(which satisfies \cref{hyp:DeRham}):
\begin{equation} \label{Discrete_deRham}
  \xymatrix{
   S_{p+1} \otimes S_{p+1} \otimes S_{p+1} \ar[r]^-{\grad} & 
   {  \left( \begin{smallmatrix}
   S_{p} &\otimes& S_{p+1} &\otimes& S_{p+1} \\ S_{p+1} &\otimes& S_{p} &\otimes& S_{p+1} \\ S_{p+1} &\otimes& S_{p+1} &\otimes& S_{p}
   \end{smallmatrix} \right) } \ar[r]^-{\curl} & 
   {  \left( \begin{smallmatrix}
   S_{p+1} &\otimes& S_{p} &\otimes& S_{p} \\ S_{p} &\otimes& S_{p+1} &\otimes& S_{p} \\ S_{p} &\otimes& S_{p} &\otimes& S_{p+1}
   \end{smallmatrix} \right) } \ar[r]^-{\Div} &
   S_{p} \otimes S_{p} \otimes S_{p} ~ 
  }
\end{equation}
where $S_p$ denotes the space of univariate splines of degree $p$. We refer to~\cite{buffa2011isogeometric} for more details. 
These spline spaces are built on a Cartesian grid of step size $h$.

\subsection{Barotropic fluid : convergence analysis}
Our first numerical test is conducted on a model of barotropic fluid which is simpler than \eqref{eqn:MHD}:
in this model the internal energy depends only of the density via $e(\rho)=\frac{\rho}{2}$ and there is no magnetic field. 

The goal of this first test is to evaluate the convergence of our scheme and verify the claimed conservations properties. For this we use a compressible Taylor-Green Vortex. The domain is a square $\Omega = [0,\pi]^2$ with periodic boundary conditions, and the initial conditions : 
$$
\begin{aligned}
\rho(x,y,0)&=1 \\
\uu(x,y,0)&=(1- 0.1\cos(2x)\sin(2y),1+0.1\cos(2y)\sin(2x)).
\end{aligned}
$$
The simulation is run until $t_f=1$, and we show the errors and convergence orders 
in \cref{tab:convergence}. For this test we used a constant time step $\Delta t = 1\E{-4}$ and a tolerance of $1\E{-8}$ for solving the non-linear system to ensure that the dominating error is due to spatial discretization. 
This demonstrates the high order accuracy of our method, with a more than optimal 
rate of convergence in this simple setting. Error is obtained by comparing with a 
reference solution computed with the same polynomial degree and $h = \pi 2^{-7}$. 
By $O(\rho_h)$ (resp. $O(\uu_h)$) we denote the order of convergence that is obtained from error $\epsilon_i$ 
and mesh size $h_i$ via $\frac{log(\epsilon_i)-log(\epsilon_{i-1})}{log(h_{i-1})-log(h_{i})}$. 
We also computed the error in the invariants and observed that for all discretizations, 
the error in the total density and total energy was never above $2\E{-12}$, which is much lower than the tolerance of the non-linear solver $1\E{-8}$, showing the robustness of our scheme.

\begin{table}
\begin{tabular}{|p{.6cm}||p{1.4cm}|p{.8cm}|p{1.4cm}|p{.8cm}||p{1.4cm}|p{.8cm}|p{1.4cm}|p{.8cm}| }
 \hline
 &\multicolumn{4}{|c||}{$p=1$} & \multicolumn{4}{|c|}{$p=2$} \\
 $h/\pi$ & $\rho_h-\rho$ & $O(\rho_h)$ & $\uu_h-\uu$ & $O(\uu_h)$ & $\rho_h-\rho$ & $O(\rho_h)$ & $\uu_h-\uu$ & $O(\uu_h)$ \\
 \hline
 $2^{-3}$ & $1.6\E{-2}$ & $ -  $ & $5.6\E{-2}$ & $ -  $ & $5.5\E{-3}$ & $ -  $ & $5.0\E{-3}$ & $ -  $ \\
 $2^{-4}$ & $5.1\E{-3}$ & $1.62$ & $1.3\E{-2}$ & $2.11$ & $3.3\E{-4}$ & $4.04$ & $3.3\E{-4}$ & $3.94$ \\
 $2^{-5}$ & $1.2\E{-3}$ & $2.06$ & $3.0\E{-3}$ & $2.10$ & $2.0\E{-5}$ & $4.06$ & $1.9\E{-5}$ & $4.06$ \\
 $2^{-6}$ & $2.5\E{-4}$ & $2.32$ & $6.1\E{-4}$ & $2.33$ & $1.2\E{-6}$ & $4.10$ & $1.1\E{-6}$ & $4.10$ \\

 \hline
\end{tabular}
\caption{Errors and convergence orders for the travelling vortex at $t=1$, for different grids and spline orders.}
\label{tab:convergence}
\end{table}

\subsection{Barotropic Kelvin-Helmholtz instability}

Our second test is a double shear layer (Kelvin-Helmholtz instability), also run with the barotropic fluid model. 
The domain for this test is a periodic rectangle $[0,1] \times [-1,1]$. The fluid is initiated with a central 
horizontal layer with positive $x$ velocity and two exterior horizontal layers with negative $x$ velocity. 
The initial density is uniform and the instability is triggered by a small perturbation in the initial $y$ velocity. 
More precisely the initial conditions are :
\begin{align*}
\rho(x,y,0) &=1 ~, \\
u_x (x,y,0) &= 0.5(T_\delta(y)-1) ~,\\
u_y (x,y,0) &= 0.1 \sin(2 \pi x) ~, \\
\text{with } T_\delta(y) &= -\tanh((y - 0.5)/\delta)+\tanh((y + 0.5)/\delta) ~,
\end{align*}
where $\delta=\frac{1}{15}$. In \cref{fig:evol_vort_Shear} we show the evolution of the vorticity for this test. 
We see that our scheme is able to produce very good results for this advection dominated test, 
even on the vorticity which is not a primal variable of the scheme, at very small scale 
(the size of the vorticity filament is comparable with the underlying grid size). 
On this more difficult test case the conservation of energy and total mass were also excellent, 
with respective variations on the order of $1\E{-13}$ and $2\E{-13}$. 

\begin{figure}
    \centering
    \includegraphics[width=0.27\textwidth]{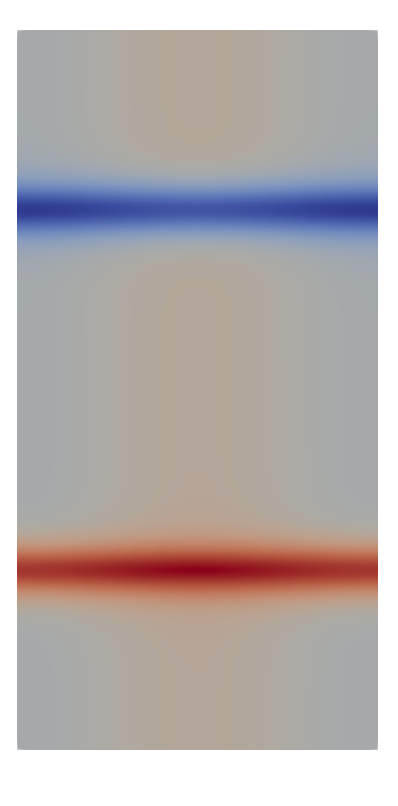}
    \includegraphics[width=0.27\textwidth]{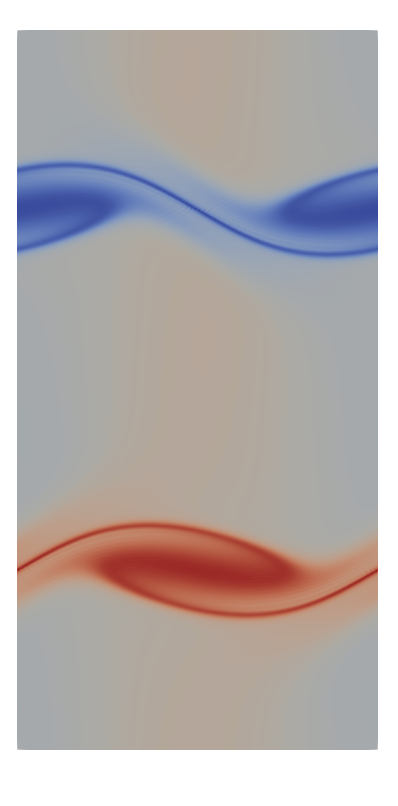}
    \includegraphics[width=0.27\textwidth]{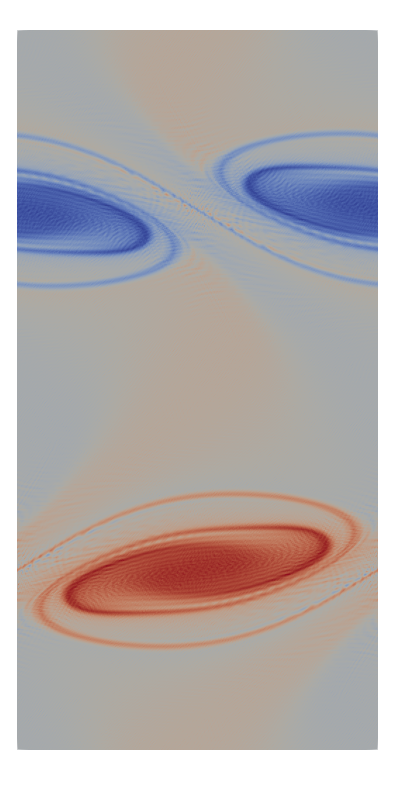}
    \includegraphics[width=0.1\textwidth]{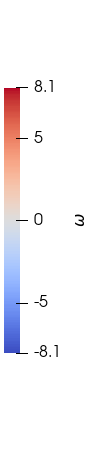}
    \caption{Vorticity at $t=0, \ 2, \ 4$, for the double shear layer simulation with a $512 \times 256$ grid,
    using the discrete spline de Rham sequence \eqref{Discrete_deRham} with $p=2$, and $\Delta t = 5\E{-4}$.}
    \label{fig:evol_vort_Shear}
\end{figure}

\subsection{Fully compressible Kelvin-Helmholtz instability}

We now shift to a more complex fluid model, namely where the energy depends on both density and entropy via 
$e(\rho,s)=\rho^{\gamma-1}e^{s/\rho}$, here with $\gamma = 7/5$. The first test we consider with this model 
is the same as the previous one, but with variable density and entropy in the initial condition and
a higher density in the central region. The entropy is also adjusted to have an initial constant pressure. 
The initial condition are therefore :
\begin{align*}
\rho(x,y,0) &= 0.5 + 0.75 T_\delta(y) ~, \\
s(x,y,0) &= -\rho(x,y,0)(\log(\gamma-1)+\gamma \log(\rho(x,y,0))) ~, \\
u_x (x,y,0) &= 0.5(T_\delta(y)-1) ~,\\
u_y (x,y,0) &= 0.1 \sin(2 \pi x) ~, \\
\text{with } T_\delta(y) &= -\tanh((y - 0.5)/\delta)+\tanh((y + 0.5)/\delta) ~,
\end{align*}
where $\delta=\frac{1}{15}$. In \cref{fig:evol_rho_FullShear} we show the evolution of the density, 
and we see that our scheme is again able to produce satisfying result even in the presence of 
more compressible effects. Advection seems to be well resolved too, as we see that the density forms a cylinder as expected. 
However we see small oscillations appearing at the end of the simulation: 
These oscillations are typically on the order of the mesh size and are probably 
due to the absence of any dissipative mechanism or upwinding in our scheme that 
could help resolve sub mesh scale phenomena. On this simulation we still have a 
very good preservation of total density and entropy ($1\E{-14}$), however the bigger 
complexity of the flow creates some (relatively small) changes in energy that might 
be imputable to the non-linear solve. The total energy varies indeed by $8\E{-7}$ 
which is the same order of the number of time step times the non-linear solver 
tolerance ($1\E{5} \times 1\E{-12}$). Thus, we can claim that invariants are very 
well preserved despite the appearing of oscillations, which indicates that we are 
reaching the limit of the smooth regime allowed by of our scheme.
\begin{figure}
    \centering
    \includegraphics[width=0.27\textwidth]{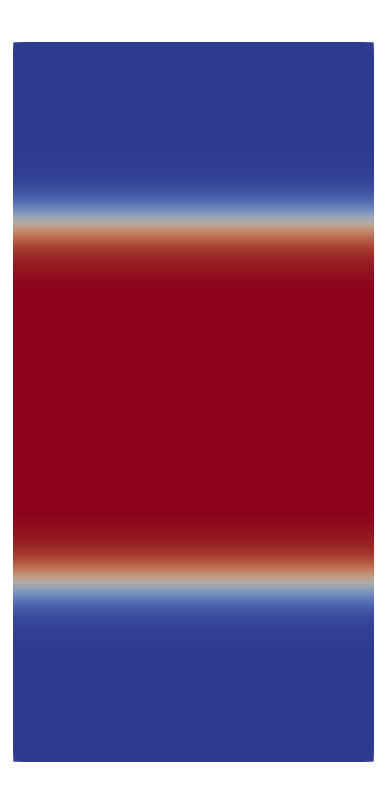}
    \includegraphics[width=0.27\textwidth]{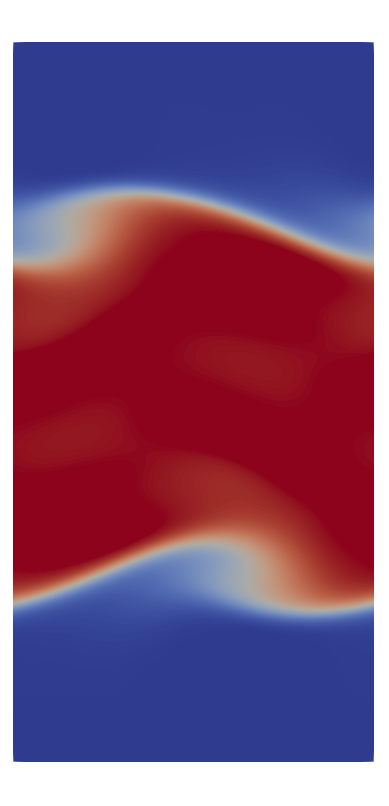}
    \includegraphics[width=0.27\textwidth]{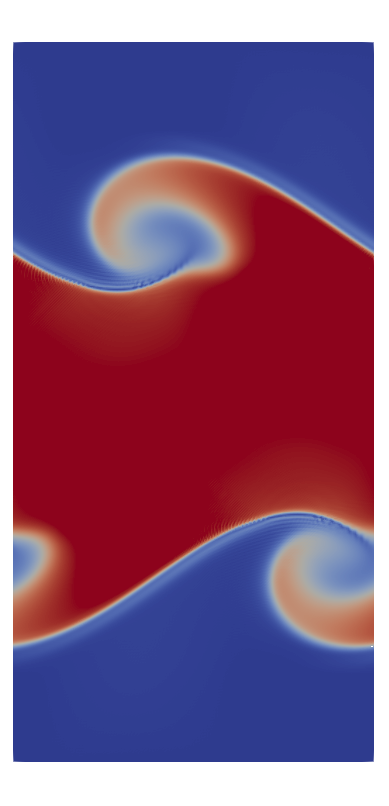}
    \includegraphics[width=0.1\textwidth]{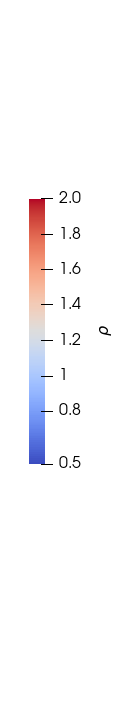}
    \caption{Density at $t=0, \ 1, \ 2$, for the fully compressible double shear layer simulation with a $512 \times 256$ grid, using the discrete spline de Rham sequence \eqref{Discrete_deRham} with degree $p=1$, 
    and $\Delta t = 2\E{-4}$.}
    \label{fig:evol_rho_FullShear}
\end{figure}
With this setup we also tested the reversibility of our method. For this purpose we ran the simulation until a final time of $t=2$, then we reversed the time step $\Delta t \mapsto -\Delta t$ and run the simulation for the same time. 
The error between the initial and final solution is then $1.6 \E{-3}$ in the velocity field and $5.9\E{-4}$ in the density field: the respective error fields are plotted in \cref{fig:FullShear_backtime}. This shows that our method 
is able to preserve the reversibility of the original system up to high accuracy, which is a fundamental property for the long time simulation of physical systems with a very low level of dissipation.
%
\begin{figure}
    \centering
    \includegraphics[width=0.35\textwidth]{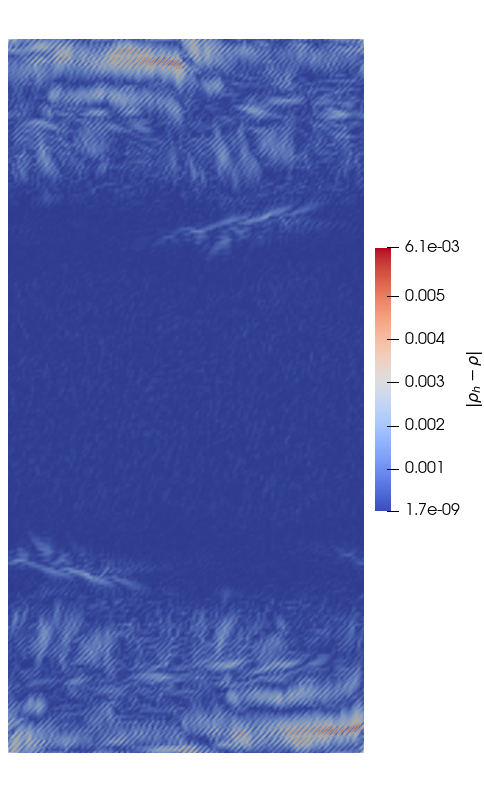}
    \includegraphics[width=0.35\textwidth]{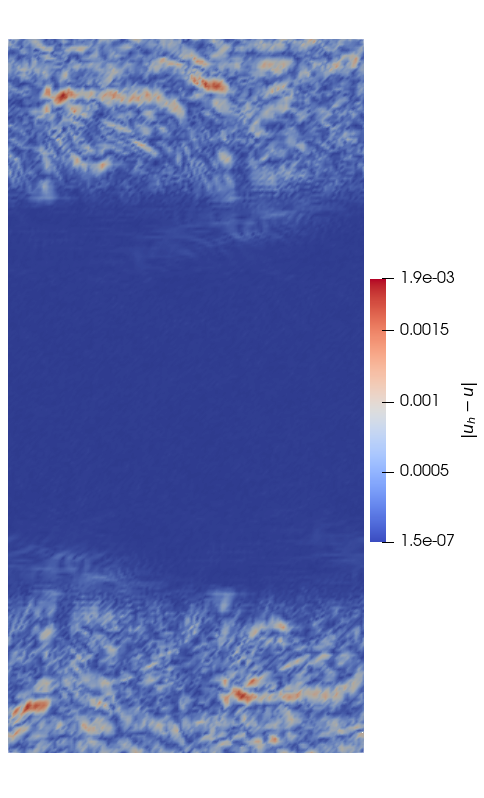}
    \caption{Error in density and velocity after reversing the scheme at $t=2$ for the fully compressible double shear layer simulation, using the spline sequence \eqref{Discrete_deRham} with degree $p=1$ on a $256 \times 128$ grid, 
    and $\Delta t = 2\E{-4}$.}
    \label{fig:FullShear_backtime}
\end{figure}
\subsection{Fully compressible fluid : Rayleigh Taylor instability}
Our second and last test with the fully compressible fluid model is a Rayleigh Taylor instability. For this simulation we added a gravitational force to our model, the Lagrangian becoming $l(\uu, \rho, s) = \int_\Omega \rho\frac{|u|^2}{2} - \rho e(\rho,s)-\rho \phi$ 
where $\phi$ is the gravitational potential, given here by $\phi = -y$ corresponding to an upward gravitational force. 
The domain is a rectangle $[0, 0.25] \times [0, 1]$ and the initial conditions are : 
\begin{align*}
\rho(x,y,0) &= 1.5 - 0.1\tanh(\frac{y-0.5}{0.02}) ~, \\
s(x,y,0) &= -\rho(x,y,0) \log(\frac{p(x,y)}{(\gamma-1)\rho(x,y,0)^\gamma}) \\
\uu(x,y,0) &= (0, -0.025 \sqrt{\frac{\gamma p (x,y)}{\rho(x,y,0)}} \cos(8 \pi x) \exp(-\frac{(y-0.5)^2}{0.09})) \\
\text{with } p(x,y) &= 1.5y + 1.25 + 0.1(0.5-y)\tanh(\frac{y-0.5}{0.02}) ~.
\end{align*}
Results are presented in \cref{fig:plot_RT1}, where one recognizes the typical shape of a Rayleigh Taylor instability. 
We are able to run this simulation without any type of upwiding which seems quite remarkable to us, 
even if this lack of stabilization might be the reason for the appearing of small instabilities at the walls. 
\begin{figure}
    \centering
    \includegraphics[width=0.15\textwidth]{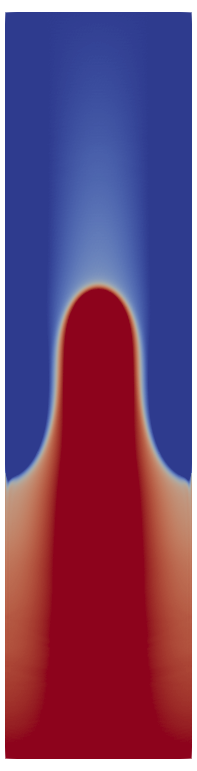}
    \includegraphics[width=0.15\textwidth]{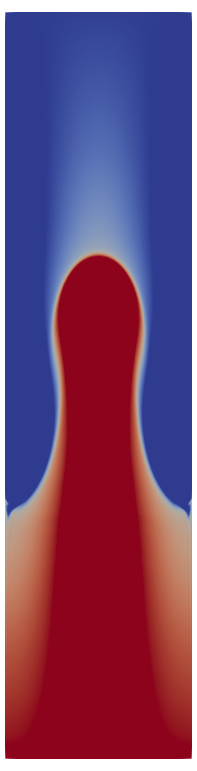}
    \includegraphics[width=0.15\textwidth]{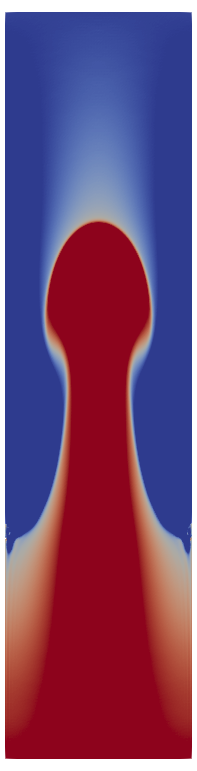}
    \includegraphics[width=0.15\textwidth]{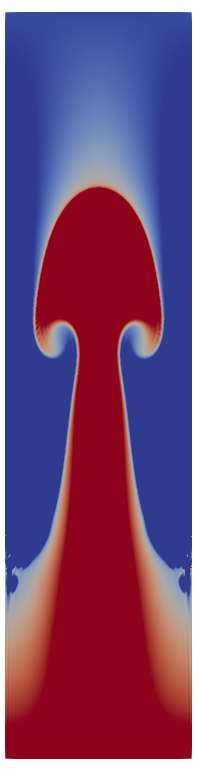}
    \includegraphics[width=0.15\textwidth]{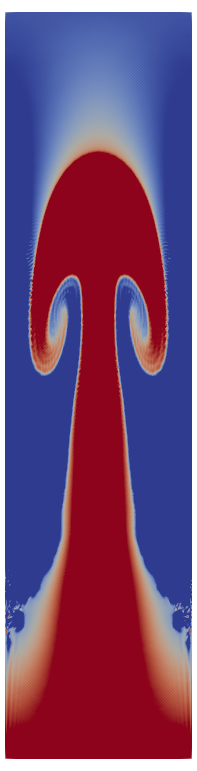}
    \includegraphics[width=0.15\textwidth]{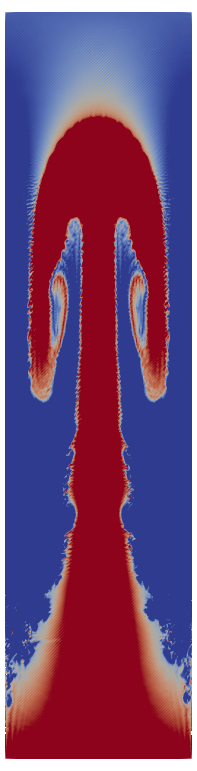}
    \includegraphics[width=\textwidth]{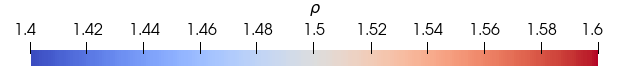}
    \caption{Evolution of the density for the Rayleigh-Taylor instability, for the times
    $t=2.0,\ 2.4,\ 2.8,\ 3.2,\ 3.6,\ 4.0$. This dissipation-less simulation 
    uses the spline spaces \eqref{Discrete_deRham} with $p=1$ on a
    $128 \times 512$ grid, and $\Delta t = 2\E{-4}$.}
    \label{fig:plot_RT1}
\end{figure}
\subsection{MHD : Alfvén wave}
We next turn to MHD simulations with a fully compressible fluid corresponding to $\gamma = 5/3$. 
Our first test is a circularly polarized Alfvén wave, where an analytical solution is known. 
The setup for this test is the following : the domain is a periodic rectangle $[0,\frac{1}{\cos \alpha}] \times [0, \frac{1}{\sin \alpha}]$ and the initial wave is given by :
\begin{align*}
\rho(x,y,0) &= 1 ~, \\
s(x,y,0) &= \log(0.1/(\gamma-1)) ~, \\
u_x (x,y,0) &= (\cos \alpha) u_{||} - (\sin \alpha) u_{\perp} ~,\\
u_y (x,y,0) &= (\sin \alpha) u_{||} + (\cos \alpha) u_{\perp} ~,\\
u_z (x,y,0) &= 0.1\cos(2\pi(x\cos \alpha+y \sin \alpha)) ~,\\
B_x (x,y,0) &= (\cos \alpha) B_{||} - (\sin \alpha) B_{\perp} ~,\\
B_y (x,y,0) &= (\sin \alpha) B_{||} + (\cos \alpha) B_{\perp} ~,\\
B_z (x,y,0) &= 0.1\cos(2\pi(x\cos \alpha+y \sin \alpha)) ~,\\
\text{with } u_{\perp} = B_{\perp} &= 0.1 \sin(2\pi (x\cos \alpha +y \sin \alpha)), ~ B_{||} = 1, ~ \ u_{||} = 0 ~.
\end{align*}
Here we have chosen $\alpha = \frac{\pi}{6}$. In \cref{fig:plot_Alven} we show the results of the simulation after one period ($t=1$). We see that the wave is very well resolved and we do not see any damping in the amplitude of phase shift, showing that our scheme is able to resolve well the coupling between $\uu$ and $\BB$. As before, density and entropy are preserved at a precision of $1\E{-14}$, energy is also very well preserved ($1\E{-11}$) and we also see in this test that the squared $L^2$ norm of the divergence of $\BB$ is of order $1\E{-27}$ which proves that the divergence preservation holds up to machine precision.

To better test our discretization we perform the same test on a longer time range of $75$ periods in \cref{fig:plot_Alven_long}. We then see some slight shifts in the wave, nevertheless the results are still very close to the initial wave, despite the rather coarse discretization ($N=16$) used here.

\begin{figure}
    \centering
    \includegraphics[width=0.49\textwidth]{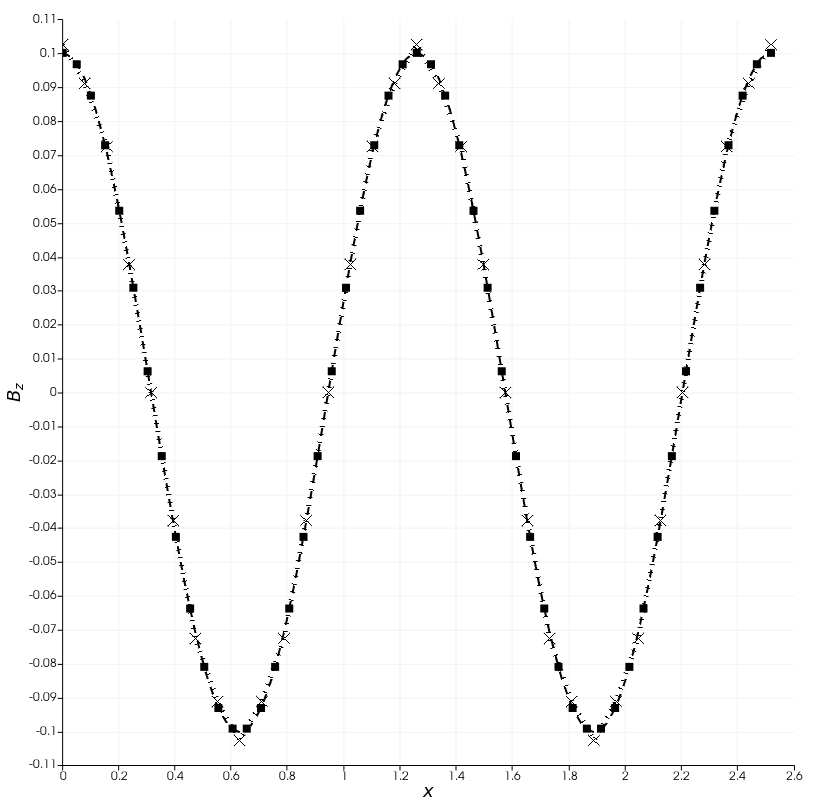}
    \includegraphics[width=0.49\textwidth]{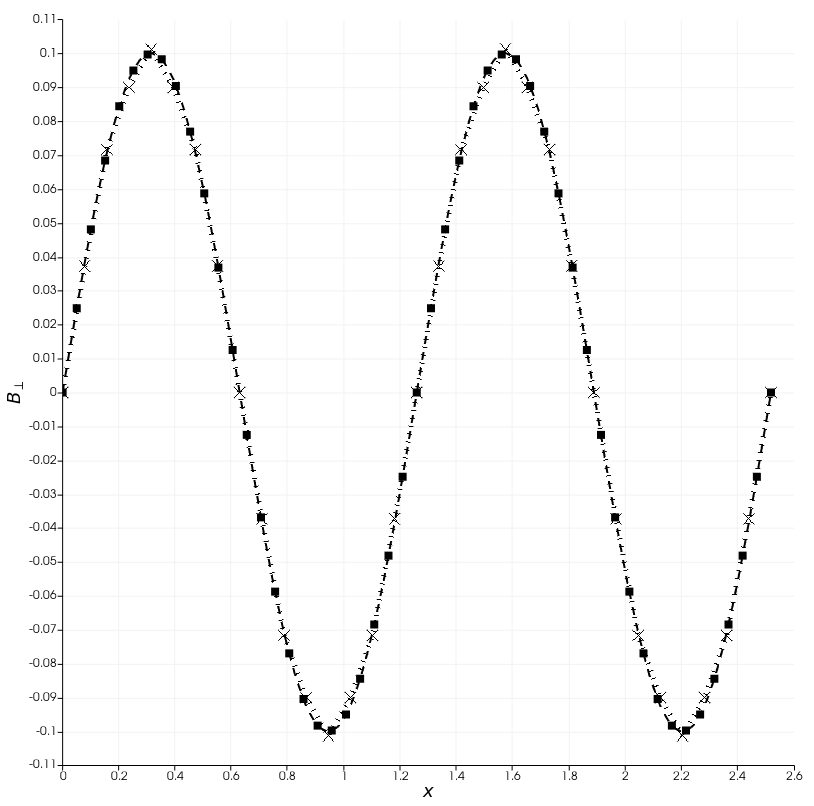}
    \includegraphics[width=0.49\textwidth]{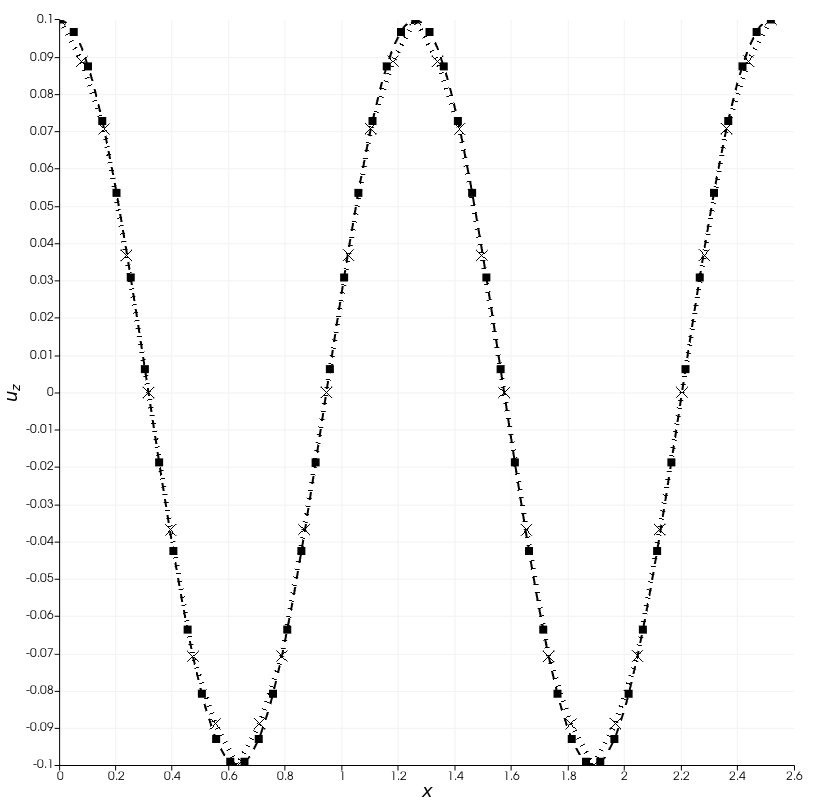}
    \includegraphics[width=0.49\textwidth]{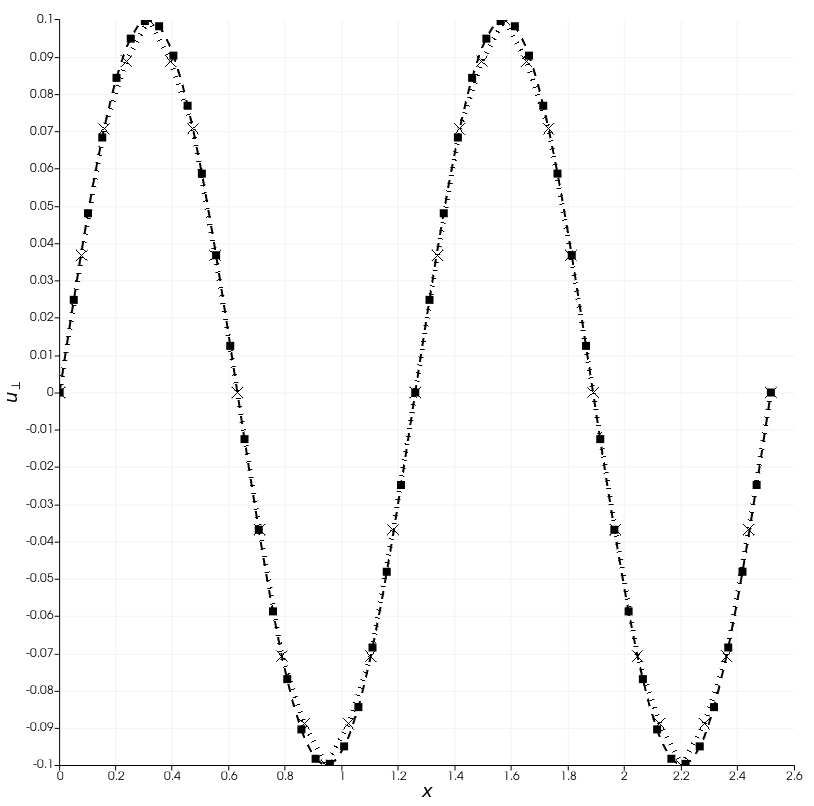}
    \caption{Plot of $v_z$, $B_z$, $v_{\perp} = \cos \alpha v_y -\sin \alpha v_x$ and $B_{\perp} = \cos \alpha B_y -\sin \alpha B_x$ for the Alfvén wave at $t=1$ (after one period of the wave), cuts at $y=0$. The squares are the reference (discretization with $N=64$ cells at $t=0$) the crosses are a coarse discretization with $N=16$ cells at $t=1$ and the dash line is a finer discretization with $N=64$ cells also at $t=1$. }
    \label{fig:plot_Alven}
\end{figure}

\begin{figure}
    \centering
    \includegraphics[width=0.49\textwidth]{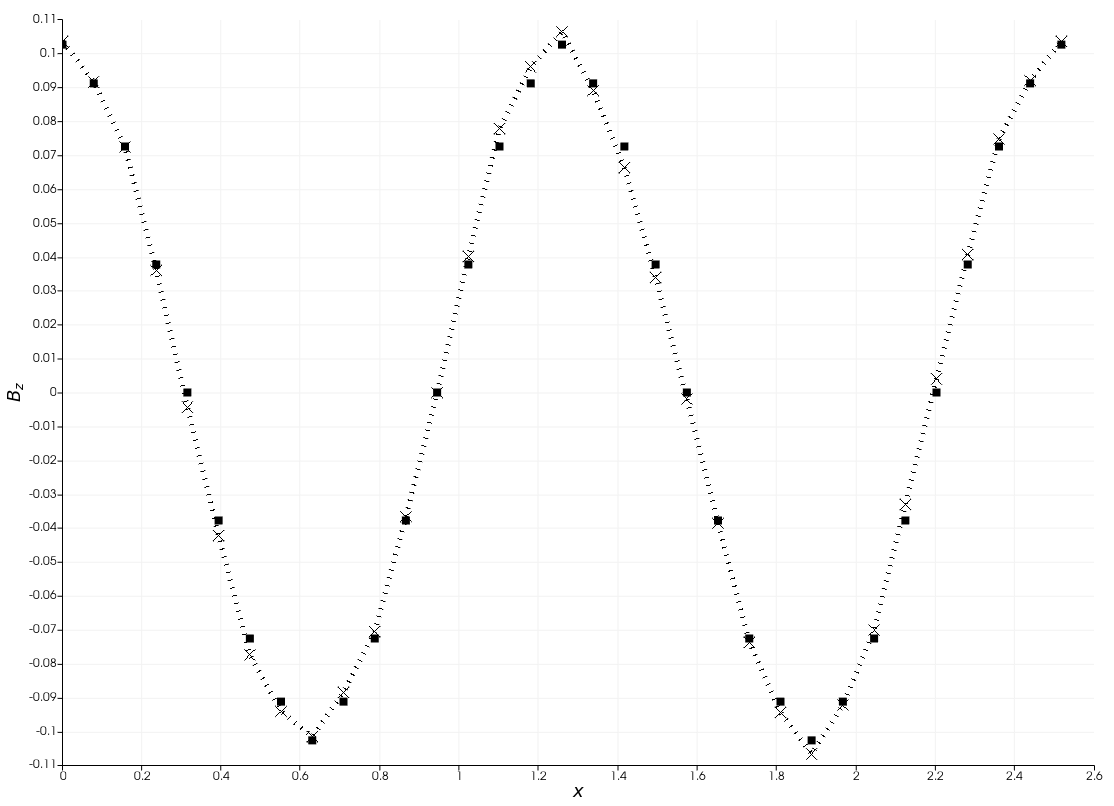}
    \includegraphics[width=0.49\textwidth]{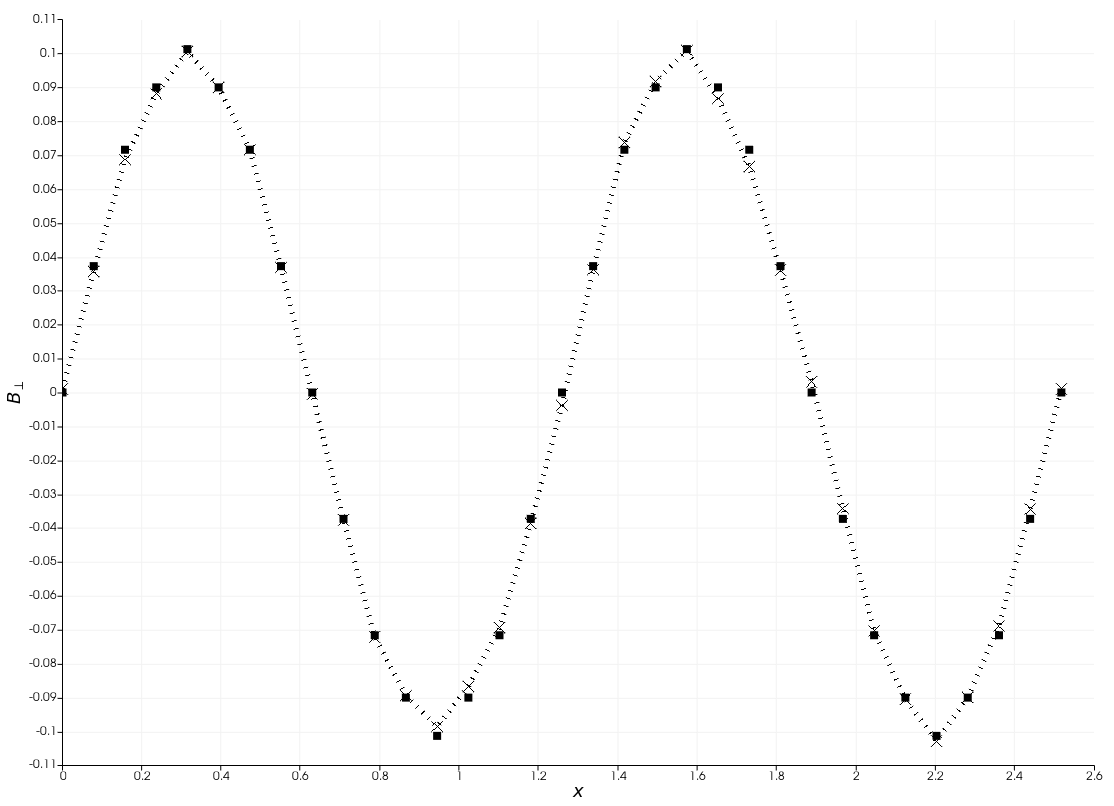}
    \caption{Plot of $B_z$ and $B_{\perp} = \cos \alpha B_y -\sin \alpha B_x$ for the Alfvén wave at $t=75$ (after 75 periods of the wave), cuts at $y=0$ with a coarse discretization ($N=16$). The squares are the reference 
    solution at $t=0$ and the crosses show the coarse solution at $t=75$.}
    \label{fig:plot_Alven_long}
\end{figure}

\subsection{MHD : Orszag-Tang vortex}
Our second MHD test is the Orszag-Tang vortex. Here the domain is a periodic box of size $[0, 2\pi]^2$ and the initial conditions are : 
\begin{align*}
\rho(x,y,0) &= \gamma^2 ~, \\
s(x,y,0) &= \gamma^2 \log\Big(\frac{\gamma}{(\gamma-1)\gamma^{2\gamma}}\Big) ~, \\
\uu (x,y,0) &= (-\sin(y), \sin(x)) ~,\\
\BB (x,y,0) &= (-\sin(y), \sin(2x)) ~ 
\end{align*}
with $\gamma = 5/3$. \Cref{fig:pressure_Orszag} presents the results of the simulations for this test. We plot the pressure at different time steps, first at $t=0.5$ where the dynamic is still smooth and we check that this pressure profile is in accordance with the results that can be found in the literature, even if the pressure is not a primary variable of our scheme and has to be recomputed from the density and the entropy. We then plot the pressure at $t=1$, soon after a shock is produced. We see that the front of the wave is still well captured, although some oscillations start to be visible in the higher pressure zone. This behaviour is due to the non-dissipative character of our scheme which prevents it to correctly resolve shocks. 
We also note that it is an expected behaviour, since our discretization relies on a variational principle which is only valid in smooth regimes.

\begin{figure}
	\centering
    \includegraphics[width=0.49\textwidth]{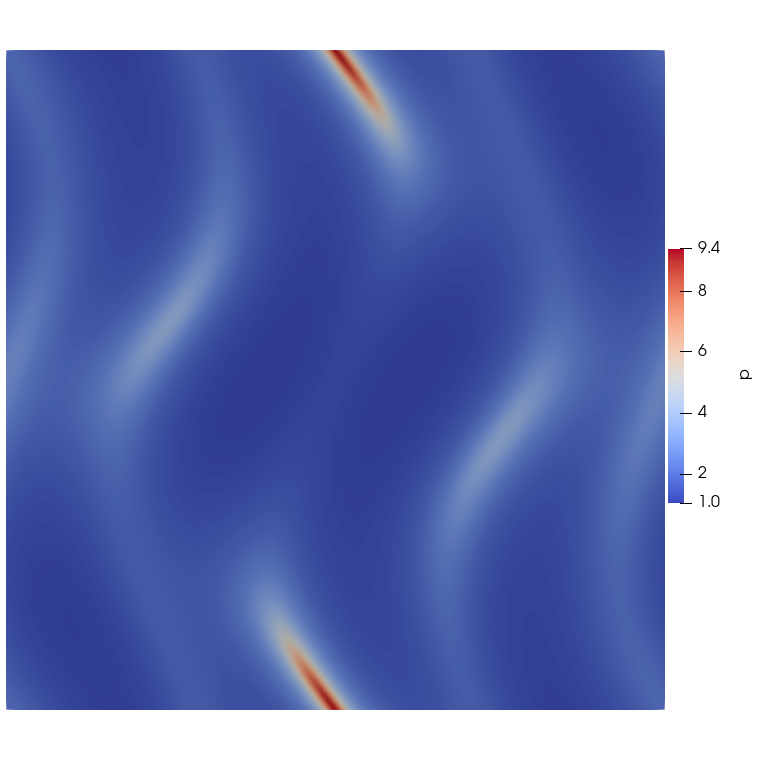}
    \includegraphics[width=0.49\textwidth]{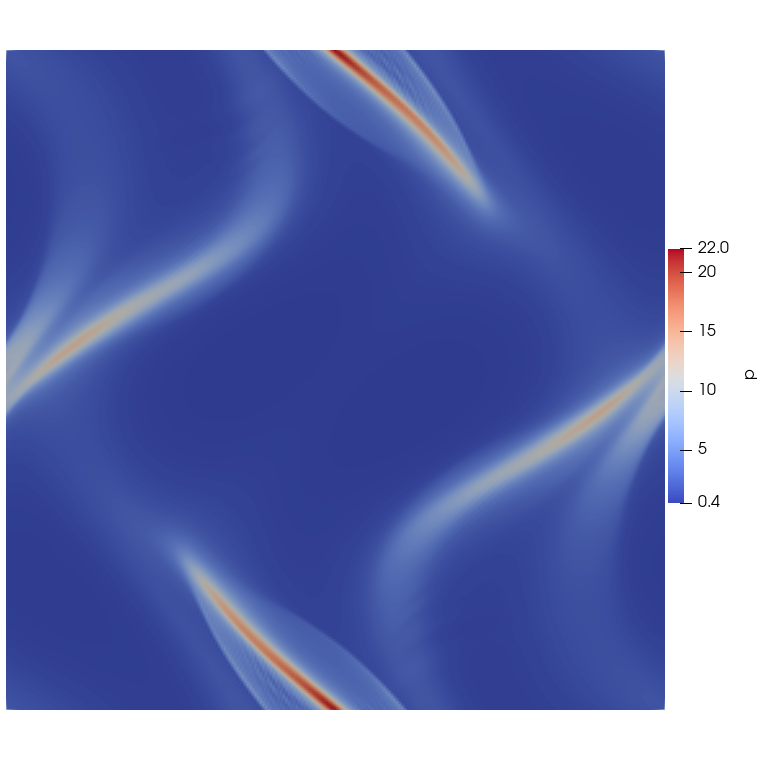}
    \caption{Plot of the pressure for the Orszag-Tang vortex at $t=0.5$ and $t=1$, 
    using spline spaces of minimal degree $p=2$ on a $256 \times 256$ grid, and a time step $\Delta t = 5\E{-4}$.}
    \label{fig:pressure_Orszag}
\end{figure}

\subsection{MHD : KHI with magnetic field}

Our final test is a magnetized Kelvin-Helmholtz instability. For this test, the domain is also a periodic rectangle $[0, 1] \times [-1, 1]$, and the initial conditions are similar to the shear tests shown above :
\begin{align*}
\rho(x,y,0) &= 1 ~, \\
s(x,y,0) &= -\log(\gamma-1) ~, \\
u_x (x,y,0) &= 0.5(T_\delta(y)-1) ~,\\
u_y (x,y,0) &= 0.1 \sin(2 \pi x) ~, \\
B_x (x,y,0) &= B_0 \\
B_y (x,y,0) &= 0. \\
\text{with } T_\delta(y) &= -\tanh((y - 0.5)/\delta)+\tanh((y + 0.5)/\delta) ~,
\end{align*}
with $\gamma = 7/5$. \Cref{fig:MKHI} shows the numerical vorticity computed for this test, and we can observe that as predicted by the theory \cite{miura1982nonlocal,berlok2019kelvin}, the solution becomes less unstable as $B_0$ increases. This last test confirms that our scheme is able to resolve complex interactions between the flow and the magnetic field, and to faithfully reproduce the instability or stability of known configurations without adding numerical dissipation that could change the results of such a study. This property makes us confident in the use of our variational approach to study inertial confinement devices where the understanding of instabilities is of primal importance.

\begin{figure}
	\centering
    \includegraphics[width=0.23\textwidth]{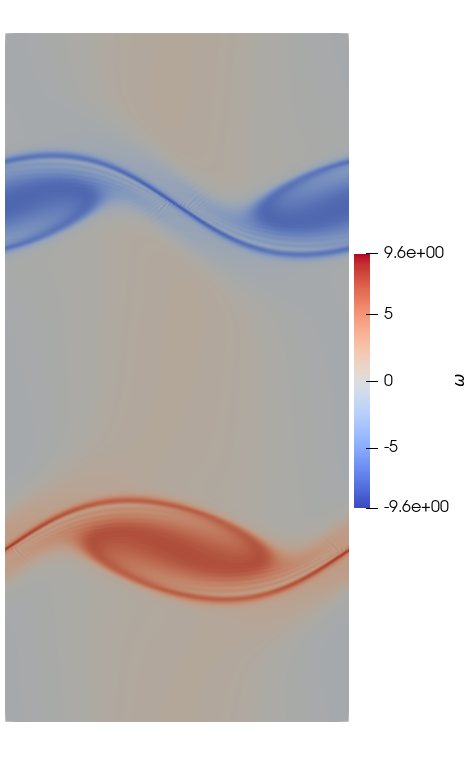}
    \includegraphics[width=0.23\textwidth]{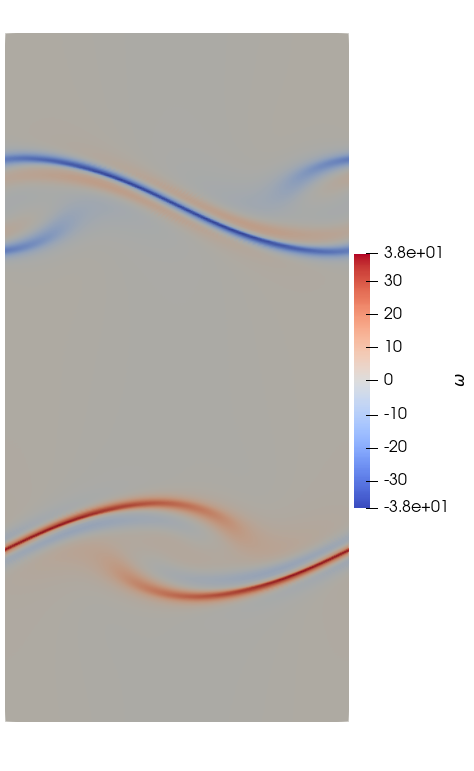}
    \includegraphics[width=0.23\textwidth]{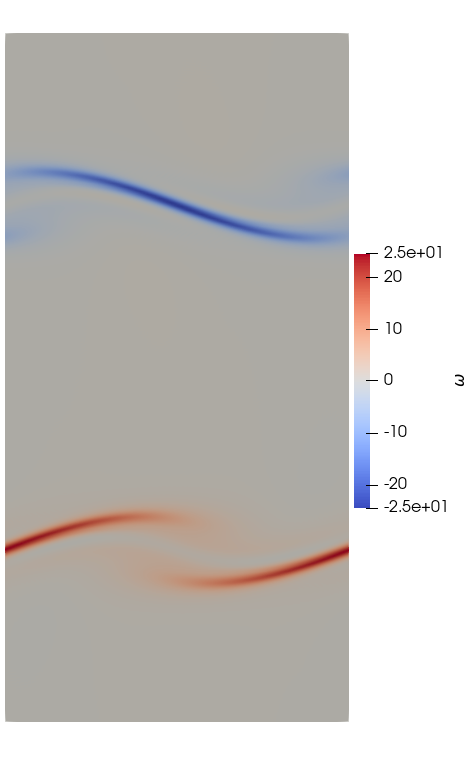}
    \includegraphics[width=0.23\textwidth]{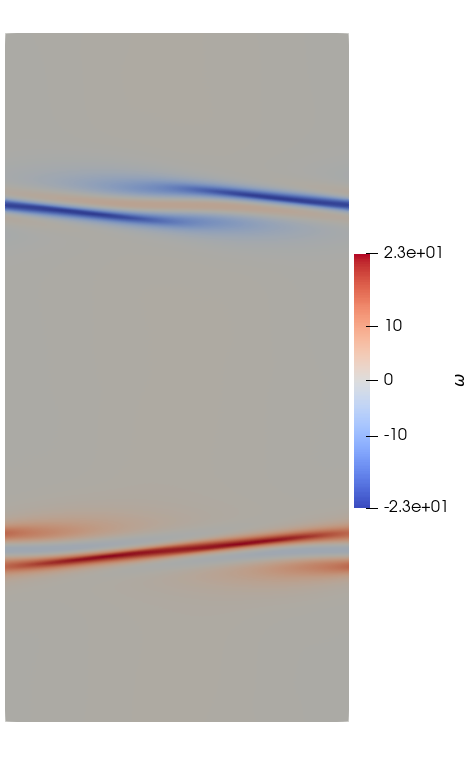}
    \caption{Plot of the vorticity at $t=2.0$ with $B_0=0,~ 0.2, ~ 0.4 \text{ and } 0.6$ for the magnetic Kelvin-Helmholtz instability,
    using spline spaces of minimal degree $p=1$ on a $256 \times 512$ grid, and a time step $\Delta t = 5\E{-4}$.}
    \label{fig:MKHI}
\end{figure}

\section{Conclusion and perspectives}
\label{sec:conclusion}
In this work we have proposed a variational discretization of the ideal MHD 
equations using the framework of Finite Element Exterior Calculus. 
Our discretization is based on transport operators built using the Cartan formula 
which allows us to closely mimic the variational principle underlying the MHD equations. 
Preservation properties of our scheme where proved first at the semi-discrete level 
and then at a fully discrete level, using a mid-point discretization in time. 
We then applied our framework on spline FEM spaces, using several fluid and MHD test cases. 
These tests have allowed us to prove the high order accuracy of our scheme, and they have also shown its 
very good preservation properties for various problems. 

Future works will focus on extending this scheme in several directions : to be able to cover more physics we will couple our MHD solver with (numerical) particles. This framework will allow us to model for instance the interaction between a bulk plasma and energetic particles in fusion devices. Being able to add some dissipation while preserving the variational structure of our scheme is also an interesting perspective to address regimes with discontinuous solutions. Finally, extending this framework to broken-FEEC spaces should allow for more flexible and local discretizations, in the scope of large simulations relying on parallel implementations.

\section*{Acknowledgements}
The authors thank the developers and maintainers from the Psydac team, in particular Yaman Güçlü, for his help in using the library for the purpose of this study and Julian Owezarek for his work on the linear algebra layer, that made this study easier. They also thank the developers of the Struphy library, in particular Stefan Possanner, for their help in understanding 
some parts of the library that were needed for this work. Inspiring discussions with François Gay-Balmaz, Evan Gawlik and Eric Sonnendrücker are also warmly acknowledged.

\bibliographystyle{elsarticle-num}
\bibliography{var-feec}

\end{document}